\documentclass[titlepage,12pt]{article}

\usepackage{array}
\usepackage[reqno]{amsmath}
\usepackage{xcolor}
\usepackage{amsthm}
\usepackage{graphicx}
\usepackage{hyperref}  
\usepackage{booktabs}
\usepackage{tikz}
\usetikzlibrary{shapes.geometric, arrows.meta, positioning}
\usepackage{mathtools,mathrsfs}
\usepackage{enumitem}
\usepackage{amssymb,amsfonts}
\usepackage{bbm}

\usepackage[
paper=letterpaper,
margin=1in,
headheight=13.6pt,
includehead,
includefoot
]{geometry}

\newcommand{\bR}{\mathbb{R}}
\newcommand{\bRn}{\mathbb{R}^n}
\newcommand{\bN}{\mathbb{N}}
\newcommand{\s}{\mathbf{s}}
\newcommand{\dom}{\mathrm{dom}}
\newcommand{\epi}{\mathrm{epi}}
\newcommand{\esp}[1]{\mathbb{E}\left[ #1\right]}
\newcommand{\espc}[2]{\mathbb{E}^{#1}\left[ #2\right]}

\def\X{{\mathbf X}}

\theoremstyle{plain}
\newtheorem{proclaim}{PROCLAIM}[section]

\newtheorem{lemma}[proclaim]{Lemma}
\newtheorem{prop}[proclaim]{Proposition}

\newtheorem{teo}[proclaim]{Theorem}

\theoremstyle{definition}

\newtheorem{assumption}[proclaim]{Assumption} 
\newtheorem{rem}[proclaim]{Remark}
\newtheorem{dfn}[proclaim]{Definition}

\begin{document}
	
	\begin{titlepage}
		\vglue 0.5cm
		\begin{center}
			\begin{large}
				{\bf On the Value Function of Convex Bolza Problems Governed by Stochastic Difference Equations}
				\smallskip
			\end{large}
			\vglue 1.truecm
			\begin{tabular}{c}
				\begin{large} {\sl Sebastian \'Alvarez$^1$, Julio Deride$^2$, Cristopher Hermosilla$^1$ } \end{large} \\
				$^1$Departamento de Matem\'atica, Universidad T\'ecnica Federico Santa Mar\'ia \\
				$^2$Facultad de Ingenier\'ia y Ciencias, Universidad Adolfo Ibáñez \\
				sebastian.alvarezo@usm.cl, julio.deride@uai.cl, cristopher.hermosill@usm.cl
			\end{tabular}
		\end{center}
		\vskip 0.25truecm
		\noindent
		{\bf Abstract}.\quad  In this paper we study the value function of Bolza problems governed by stochastic difference equations, with particular emphasis on the convex non-anticipative case. Our goal is to provide some insights on the structure of the subdiferential of the value function. In particular, we establish a connection between the evolution of the subgradients of the value function and a stochastic difference equation of Hamiltonian type. This result can be seen as a transposition of the method of characteristics, introduced by Rockafellar and Wolenski in the 2000s, to the stochastic discrete-time setting. Similarly as done in the literature for the deterministic case, the analysis is based on a duality approach. For this reason we study first a dual representation for the value function in terms of the value function of a dual problem, which is a pseudo Bolza problem. The main difference with the deterministic case is that (due to the non-anticipativity) the symmetry between the Bolza problem and its dual is no longer valid.  This in turn implies that ensuring the existence of minimizers for the Bolza problem (which is a key point for establishing the method of characteristics) is not as simple as in the deterministic case, and it should be addressed differently. To complete the exposition, we study the existence of minimizers for a particular class of Bolza problems governed by linear stochastic difference equations, the so-called linear-convex optimal control problems.
		.
		\vskip 0.25truecm
		\halign{&\vtop{\parindent=0pt
				\hangindent2.5em\strut#\strut}\cr
			{\bf Keywords}: Convex Bolza problems, Stochastic Difference Equations, Hamiltonian Systems, Discrete-time systems, Existence of minimizers\hfill\break\cr
			{\bf Date}: November 20, 2025 \cr}
	\end{titlepage}

	\section{Introduction}
	This paper is concerned with Bolza problems governed by \emph{stochastic difference equations}, where the underlying probability space is given by $(\Omega,\mathscr{A},\mu)$.  Specifically, for fixed integers $\tau <T$, we study the optimization problem
	\begin{align}
		\label{eq:StochBolza}\tag{P}
		\text{Minimize }\esp{~\sum_{t=\tau+1}^TL_t(\omega,x_{t-1},\Delta x_t )}+g\left(\esp{x_T}\right),\end{align}
	subject to the constraint $\esp{x_\tau} = \xi$, where $\xi \in \bRn$ is the (expected) initial state, treated here as an input of the system; the parameter $\tau$ is interpreted as the initial time, and $T$ as the final horizon. The random variable $x_t$ represents the configuration of the system at time $t$, and $\Delta x_t := x_t - x_{t-1}$ denotes its increment between two consecutive time instants. Here and in what follows, we will omit the dependence on the events $\omega\in\Omega$ of the random variables, notably $x_t$ and $\Delta x_t$ in~\eqref{eq:StochBolza}. The decision vectors $x_{\tau+1},\ldots,x_T$ belong to appropriate function spaces that ensure the problem is well-defined and that no measurability nor integrability issues arise.

	The focus of this work is on \emph{non-anticipative} stochastic Bolza problems of convex type, where the stage costs $L_{\tau+1},\ldots,L_T$ are  convex normal integrands (c.f.~\cite[16.D]{RockafellarWets1998}), and the terminal cost $g$ is a convex function; not necessarily finite everywhere. The \emph{non-anticipativity} approach we take means that, in the decision-making process, the decision variables $x_t:\Omega\to\bRn$ depends only on information available at time $t$, which is represented by a $\sigma$-field $\mathscr{G}_{t}$, and not on future random events, that is, \[\mathscr{G}_{\tau} \subset\mathscr{G}_{\tau+1}\subset\ldots \subset\mathscr{G}_{T}\subset \mathscr{A}.\]

	The central object we study is the so-called \emph{value function} of the problem, which is defined as the optimal value of the problem when considering the initial state $\xi$ (at time $t=\tau$) as an input variable. The key feature we exploit in this paper is the following. If, for every event $\omega\in\Omega$, the mappings $L_{\tau+1}(\omega,\cdot,\cdot),\ldots,L_{T}(\omega,\cdot,\cdot)$ and the terminal cost $g$ are convex, then the value function is itself a convex function on the initial state. Notably, in this setting, the Lagrangians are jointly convex in the state and velocity variables at the same time. This property provides the foundation for a duality-based analysis of convex Bolza problems. In particular, one of our intermediate results is to derive a dual representation for the value function in terms of the Fenchel conjugate, similarly as done in \cite{RocWol00,HerWol19} for continuous-time deterministic systems, in \cite{BarSir2019} for continuous-time stochastic systems or in \cite{deride2024subgradientevolution} for discrete-time deterministic systems. It is important to highlight that, in the deterministic settings, the primal and dual problems are completely symmetric Bolza problems (they have the same structure), however this is no longer the case for the non-anticipative stochastic framework.
	
	Our main goal is to investigate how the subgradients of value function evolve over time and how this evolution can be interpreted in terms of stochastic difference equations of Hamiltonian type, in a manner analogous to what has been done for continuous-time systems in \cite{RocWol00,HerWol19} or for discrete-time systems in \cite{deride2024subgradientevolution}. In \cite[Theorem~2.4]{RocWol00}, Rockafellar and Wolenski introduced a global method of characteristics for convex Bolza problems governed by continuous-time systems, which describes the subgradient dynamics of the value function via trajectories of an associated Hamiltonian system (an ODEs system). A key step required for establishing the method of characteristics is being able to ensure the existence of primal and dual optimal solutions. This usually is accomplished by assuming a \emph{Slater-type} qualification condition. Enforcing this \emph{interior feasibility qualification condition} on the primal, implies that the dual problem has a solution, and vice versa, since the primal and dual problem are symmetric. Something analogous happens in the deterministic  discrete-time case.
	
	For deterministic discrete-time optimal control (a particular type of Bolza problem), a comparable method of characteristics appears in \cite[Proposition~5.13]{de2008sustainable}, but with rather strong assumptions: the value function and the problem data are smooth, the minimizer is unique, and no state nor control constraints are allowed. A more general case was studied in \cite{deride2024subgradientevolution}, where the authors present a method of characteristic, which requires rather mild assumptions. Indeed, as mentioned earlier and thanks to the symmetry between the primal and dual problems, symmetric strict interior feasibility conditions on the primal and dual problems are enough for the analysis. Here we perform a similar study as in \cite{deride2024subgradientevolution}, but for problems governed by stochastic difference equations. In this setting, the non-anticipativity breaks down the symmetry between primal and dual problems, and so the Slater-type condition only ensures the existence of a dual optimal solution. Enforcing a interior feasibility qualification condition on the dual problem is possible, and it could eventually lead to the existence of a primal minimizer. However, because of the complex structure of the dual problem, the verification of such qualification condition is more involved. For this reason, in our approach, we establish the method of characteristics by assuming first the existence of primal solutions, and later we study a particular case (the so-called linear-convex optimal control problem), where the existence of minimizers can be guaranteed under rather mild assumption on the data of the primal problem, and without need of actually computing the dual problem.

	To the best of our knowledge, no prior work has developed a duality-based approach to study either the value function of stochastic discrete-time convex Bolza problems or the evolution of its subgradients. Recent work on the dynamic programming principle for general convex stochastic optimization problems has focused on characterizing the value functions recursively via the conditional expectation of normal integrands, extending the initial abstract formulations by Rockafellar and Wets; see \cite{PenPer23,PenPer25}. For a general discussion of value functions in stochastic discrete-time optimal control, and their relation to the dynamic programming equation, we refer to \cite[Ch.~8]{de2008sustainable}. 
	
	Note that convex Bolza problems with extended-valued Lagrangians allow us to recover well-known control problems such as the linear quadratic regulator or the linear-convex problems with mixed constraints. This has been discussed in \cite[Ex.2]{Roc70} and \cite{Roc87,GoeSub07,HerWol17,BecHer22} for continuous-time systems; the arguments can be readily adapted to a discrete-time setting.  We also mention that convex Bolza problems have been intensively studied in the context of continuous-time systems.  Indeed, in the 1970s, in a series of papers \cite{Roc70,Roc71,Roc72,Roc73,Roc76}, Rockafellar laid the foundations of a duality theory for this type of problems; \cite{RocWet83} is a transposition of the theory to discrete-time problems, and \cite{Bis73} is one to stochastic time-continuous problems.  The value function is studied via the duality theory in \cite{Goe04b,Goe04a,Goe05,Goe05a,HerWol17,HerWol19,Roc04,RocGoe08,RocWol00,RocWol00b}, in the absence of pathwise constraints.
	
	The main contributions of this paper are the following. On the one hand, we transpose the method of characteristics to discrete-time stochastic systems, by showing that the the subgradients of the value function can be obtained by mean of backward trajectories of a stochastic discrete-time Hamiltonian system (a stochastic difference equation) coupled with an appropriate transversality condition. For this reason, we establish first a dual representation for the value function in terms of the Fenchel conjugate of the value function of a suitable dual problem. Although the methodology we use in this paper is similar to the one used in \cite{deride2024subgradientevolution},  the technical details differ in many aspects. As pointed out before, the main difference with the deterministic case presented in \cite{deride2024subgradientevolution}, is that the symmetry between the primal and dual problems is no longer valid. To complete the analysis, we study the existence of minimizers for a particular class of Bolza problems governed by linear stochastic difference equations, the so-called linear-convex optimal control problems.

	\subsection{Notation and essentials}
	Throughout this article we use the following notation: $|\cdot|$ denotes the Euclidean norm and $a \cdot b$ stands for the Euclidean inner product of two vectors $a,b\in\bRn$. We set $[\![p:q]\!]:=\{p,p+1,\ldots,q\}$, that is, it is the collection of all integers between $p$ and $q$ (inclusive), assuming always that $p\leq q$; if $p=q$, we set $[\![p:q]\!]=\{p\}$. The (convex) normal cone to $S$ at $x\in S\subset \bRn$ is the set
	\[\mathcal{N}_S(x):=\{z\in\bRn\mid\ z\cdot(s-x)\leq 0,\ \forall s\in S\}.\]
	
	For probability space $(\Omega,\mathscr{A},\mu)$, we write $L^p(\Omega,\mathscr{A},\mu: \bRn)$  for the (equivalence class of) $\mathscr{A}$-measurable $\bRn$-valued functions defined on $\Omega$ that are $p$-integrable if $p<+\infty$ or essentially bounded if $p=+\infty$. We say that an $\mathscr{A}$-measurable real-valued function defined on $\Omega$ is \emph{summable} when it is integrable with finite integral.

	Let $X$ be a topological vector space, suppose $\varphi: X\to\bR\cup\{+\infty\}$ is a function.  The effective domain of $\varphi$ is the set 
	\[\dom(\varphi):=\{x\in X\mid\ \varphi(x)<+\infty\}.\]  
	The function $\varphi$ is said to be {\it proper} if $\dom(\varphi)\neq\emptyset$ and $\varphi(x)>-\infty$ for all $x\in X$;  {\it convex} if $\epi(\varphi):=\{(x,r)\in X\times\bR\mid\ \varphi(x)\leq r\}$ is a convex set, and  {\it lower semicontinuous (l.s.c.~for short)} if $\epi(\varphi)$ is a closed set in $X\times\bR$ in the usual product topology.
	
	If  $Y$ is another topological vector space and $\langle\cdot,\cdot\rangle$ is the duality pairing between $Y$ and $X$, the conjugate of $\varphi:X\to\bR\cup\{+\infty\}$ is the mapping $\varphi^*:Y\to\bR\cup\{\pm\infty\}$ defined via the formula
	\[\varphi^*(y):=\sup\left\{\langle y,x\rangle-\varphi(x)\mid\ x\in X\right\},\qquad\forall y\in Y,\]
	and its subdifferential  at $x\in\dom(\varphi)$ is the set
	\[\partial \varphi(x):=\{y\in Y\mid \varphi(x)+\langle y,z-x\rangle \leq \varphi(z),\ \forall z\in X\}.\]
	
	These mathematical objects are related via the Fenchel-Young equality: 
	\begin{equation}\label{eq:FYeq}
		y\in\partial \varphi(x)\quad\Longleftrightarrow\quad \varphi(x)+\varphi^*(y)=\langle x, y\rangle.
	\end{equation}
	For the case $X=L^{\infty}(\Omega,\mathscr{A},\mu: \bRn)^m$ and $Y=L^1(\Omega,\mathscr{A},\mu: \bRn)^m$ for some $m\in\bN$, we consider duality pairing given by:
	\[\langle y, x \rangle = \esp{\sum_{t=1}^{m} y_{t} \cdot x_{t}}.\]
	A function $h:\bRn\times\bRn\to\bR\cup\{+\infty\}$ is called \emph{concave-convex} if 
	$h_y(\cdot)=-h(\cdot,y)$ and $h_x(\cdot)=h(x,\cdot)$ are convex.
	The (concave-convex) subdifferential of $h$ is the set
	\begin{equation}\label{eq:subdif_saddle}
		\partial h(x,y):=[-\partial h_y(x)]\times\partial h_x(y),\qquad\forall (x,y)\in\bRn\times\bRn.	
	\end{equation}

	\section{Stochastic Convex Bolza Problems}\label{sec:duality}
	The goal of this paper is to study non-anticipative stochastic Bolza problems as \eqref{eq:StochBolza}, associated with stage costs $L_{\tau+1},\ldots, L_{T}: \Omega\times\bRn\times \bRn\rightarrow \mathbb{R} \cup \{+\infty\}$ (Lagrangians) and a given terminal cost $g:\bRn\rightarrow \mathbb{R}\cup\{+\infty\}$. The non-anticipativity is represented by means of a family of $\sigma$-field   $\mathscr{G}=\{\mathscr{G}_{t}\}_{t=\tau}^{T}$, associated with the underlying probability space $(\Omega,\mathscr{A},\mu)$, such that $\mathscr{G}_{t}\subset \mathscr{G}_{t+1}\subset\mathscr{A}$ for each $t \in [\![\tau:T-1]\!]$. 
	
	The main object we study is the so-called \emph{value function} issued from problem \eqref{eq:StochBolza}, which corresponds to the mapping that assigns to any initial state $\xi \in \bRn$, and $\s \in [\![\tau: T-1]\!]$, the optimal value of the problem. In mathematical terms it is given by \begin{equation}\label{primalvaluefunction} \mathbf{V}_{\s}(\xi):=\inf_{\left\{x_t\right\}_{t=\s}^T \in \mathscr{N}_{\s} } \left\{ \esp{~\sum_{t=\s+1}^T L_t\left(\omega, x_{t-1}, \Delta x_t\right)}+g\left(\esp{x_T}\right)\mid \esp{x_{\s}}=\xi\right\}.    
	\end{equation} 
	Recall that $\Delta x_t := x_t - x_{t-1}$. Here,  $\mathscr{N}_{\s}$ stands for set of feasible processes adapted to the filtration $\mathscr{G}$, that is, the (convex) subset of $\mathscr{L}^{\infty}_\s:=L^{\infty}(\Omega,\mathscr{A},\mu;\bRn)^{T-s+1}$ given by
	\[\mathscr{N}_{\s}=\left\{\left\{x_t\right\}_{t=\s}^T\in \mathscr{L}^{\infty}_{\s}\mid x_{t}\text{ is } \mathscr{G}_{t}-\text{measurable for any }t \in [\![\s:T]\!]\right\}.\]
	Our main concern is the convex case. Therefore, and unless otherwise explicitly stated, we assume the following basic conditions:
	\begin{equation}\tag{A}\label{hyp:convex}
		\begin{cases}
			L_{t}\text{ is a  normal convex integrand for each $t\in [\![\tau: T-1]\!]$;}\\
			g\text{ is a proper, convex and l.s.c.~function.}
		\end{cases}
	\end{equation}
	Recall that a function $\varphi:\Omega\times X\to\bR\cup\{+\infty\}$ is said to be a \emph{normal convex integrand} if $\omega \mapsto \operatorname{epi}(\varphi(\omega,\cdot))$ is an $\mathscr{A}$-measurable\footnote{A set-valued map $F:\Omega\rightrightarrows \bR\times\bRn$ is said to be $\mathscr{A}$-measurable if for any open set $O\subset\bR\times\bRn$ we have that $F^{-1}(O)\in \mathscr{A}$.} set-valued map with nonempty, convex and closed values; in \eqref{hyp:convex} we use $X=\bRn\times\bRn$.
	In particular, it follows that for each $t\in [\![\tau: T-1]\!]$ and  each event $\omega\in\Omega$ fixed, the mapping $L_{t}(\omega,\cdot,\cdot)$ is a  proper, convex and l.s.c.~function. A rather basic, but important, consequence of \eqref{hyp:convex} is that \[\omega\mapsto \sum_{t=\s+1}^T L_t\left(\omega, x_{t-1}, \Delta x_t\right)\] is an  $\mathscr{A}$-measurable function for any $\left\{x_t\right\}_{t=\s}^T \in \mathscr{L}^{\infty}_\s$; c.f.~\cite[Prop.~14.28]{RockafellarWets1998}.
	
	Moreover, to ensure that the expected value in the cost to be minimized in \eqref{eq:StochBolza} is well defined (or takes the value $+\infty$), we will impose the following integrability condition.
	\begin{equation}\tag{B}\label{hyp:integrability}
		\begin{cases}
			\forall\rho>0,~\exists\gamma:\Omega\rightarrow\bR,\text{ a summable function such that}\\
			L_{t}(\omega,x,v)\geq \gamma \text{ a.s. whenever }|x|,|v|\leq \rho,~ t\in[\![\tau+1:T]\!.]
		\end{cases}
	\end{equation}
	We will also assume that \eqref{hyp:integrability} is in force all along this paper. Therefore, for any $\left\{x_t\right\}_{t=\s}^T \in \mathscr{L}^{\infty}_\s$, there is a summable function $\tilde\gamma:\Omega\rightarrow \bR$ such that
	\[\esp{~\sum_{t=\s+1}^T L_t\left(\omega, x_{t-1}, \Delta x_t\right)}\geq (T-\s)\esp{\tilde\gamma}>-\infty.\]
	
	A (non surprising) consequence of the structural assumption \eqref{hyp:convex} is the convexity of value function.
	\begin{prop}
		The value function $\mathbf{V}_{\s}$ is convex for any $\mathrm{s}\in [\![\tau: T-1]\!]$.
	\end{prop}
	\begin{proof}
		Let $\mathrm{s} \in [\![\tau: T-1]\!]$. It suffices to notice that the function:
		$$
		\mathcal{L}_s\left(\left\{x_t\right\}_{t=\s}^T, \xi\right):= \begin{cases}\esp{~\sum_{t=\s+1}^T L_t\left(\omega, x_{t-1}, \Delta x_t\right)}+g\left(\esp{x_T}\right), & \text { if } \esp{x_{\s}}=\xi, \\ +\infty & \text { otherwise },\end{cases}
		$$    
		is convex, and that $\mathbf{V}_{\s}(\xi)=\inf \left\{\mathcal{L}_{\s}\left(\left\{x_t\right\}_{t=\s}^T, \xi\right) \mid\left\{x_t\right\}_{t=\s}^T \in \mathscr{N}_{\s}\right\}$.
	\end{proof}

	\subsection{A Dual Representation for the Value Function}
	
	The convex framework considered in this paper prompts for a duality theory that can be derived through the standard perturbation function approach (see, e.g., \cite{RocWet83}).  Inspired by the analysis in \cite{deride2024subgradientevolution} for the deterministic version of this problem and the theory developed in \cite{RocWet83}, we now consider a dual problem to \eqref{eq:StochBolza}. It is worthy to mention that, since the state of the primal problem evolves in a subset of $\mathscr{L}^{\infty}_\s$, it is rather natural that the state of the dual problem, which is a pseudo Bolza problem (it has not exactly the same structure as \eqref{eq:StochBolza}), evolves in a space that can be put in duality with $\mathscr{L}^{\infty}_\s$. A suitable candidate for such task is $\mathscr{L}^1_{\s}=L_{1}(\Omega,\mathscr{A},\mu;\bRn)^{T-s+1}$. Furthermore, as in the deterministic case, the stage costs (dual Lagrangians) and  (dual) terminal cost are obtained by means of the Fenchel conjugate. To be more precise, we set $M_{\tau+1}, \ldots, M_T: \Omega\times \bRn \times \bRn \rightarrow \mathbb{R} \cup\{+\infty\}$ as the dual Lagrangians defined as
	\[M_t(\omega, p, q):=\sup _{x, v \in \bRn}\left\{x \cdot q+v \cdot p-L_t(\omega,x, v)\right\}, \quad \forall p, w \in \bRn, \forall t \in [\![\tau: T-1]\!],\]
	and $f: \bRn \rightarrow \mathbb{R} \cup\{+\infty\}$ as the dual terminal cost given by
	\[f(b):=g^*(-b), \quad \forall b \in \bRn.\]
	
	Therefore, the value function of the dual problem to \eqref{eq:StochBolza} is given, for any $\eta \in \bRn$ and $\mathrm{s} \in [\![\tau : T-1]\!]$, by the following expression:
	\begin{equation}\label{dualvaluefunction}
		\mathbf{W}_{\s}(\eta):=\inf_{\{p_{t}\}_{t=\s}^{T}\in\mathscr{P}_{\s}}\left\{\esp{\sum_{t=\s+1}^{T}M_{t}\left(\omega,\espc{t}{p_{t}},\espc{t}{\Delta p_{t}}\right)} +f(p_{T})\mid \esp{p_{\s}}=-\eta\right\}
	\end{equation}
	Here, $\espc{t}{\cdot}$ stands for the (regular) \emph{conditional expectation} with respect to $\mathscr{G}_{t}$ and $\mathscr{P}_{\s}\subseteq \mathscr{L}^1_{\s}$ is the set given by 
	\[\mathscr{P}_{\s}=\left\{(p_{\s},\ldots,p_{T})\in \mathscr{L}^1_{\s}\mid \begin{matrix}
		p_{t-1}\text{ is } \mathscr{G}_{t}-\text{measurable for }t \in [\![\s+1:T]\!],\\ \espc{\s}{p_{\s}} \text{ and }p_{T} \text{ are constant functions on }\Omega
	\end{matrix}
	\right\}.\]
	
	Note that, as for the primal problem, $\Delta p_{t}=p_{t}-p_{t-1}$. At this point we also remark that, from the theory of convex normal integrands  (c.f.~\cite{RockafellarWets1998}), it follows that $M_{\tau+1}, \ldots, M_T$ are  normal convex integrands, and thus,
	\[\omega\mapsto \sum_{t=\s+1}^{T}M_{t}\left(\omega,\espc{t}{p_{t}},\espc{t}{\Delta p_{t}}\right)\] is an  $\mathscr{A}$-measurable function for any $\left\{p_t\right\}_{t=\s}^T \in \mathscr{L}^1_\s$.

	From now on, we will refer to $\mathbf{V}_{\s}$ given by \eqref{primalvaluefunction} as the \emph{primal value function}, while $\mathbf{W}_{\s}$ given by \eqref{dualvaluefunction} will be called  the \emph{dual value function}.
	
	Similarly as for the continuous-time case, we can extend the definition of both value functions up to time $\s=T$, as follows:
	$$
	\mathbf{V}_T(\xi)=g(\xi) \quad \text { and } \quad \mathbf{W}_T(\eta)=f(-\eta), \quad \forall \xi, \eta \in \bRn.
	$$
	From the definition of the conjugate, it follows that
	$$
	\mathbf{V}_T(\xi)+\mathbf{W}_T(\eta) \geq \xi \cdot \eta, \quad \forall \xi, \eta \in \bRn .
	$$
	Notice that in particular, we have that $\mathbf{W}_T\geq\mathbf{V}_T^*$ and $\mathbf{V}_T\geq\mathbf{W}_T^*$. This weak duality relation becomes strong (with equality) whenever we enforce the transversality condition \(-\eta\in \partial g(\xi)\); this is a straightforward consequence of Fenchel-Young equality \eqref{eq:FYeq}. The next proposition shows that this weak duality relation propagates backward in time. Later on, we will show that the weak duality stated below becomes strong whenever the primal problem is qualified and has a minimizer.
	\begin{prop}\label{prop:weak}
		For any $\s \in [\![\tau: T-1]\!]$, we have that
		$$
		\mathbf{V}_{\s}(\xi)+\mathbf{W}_{\s}(\eta) \geq \xi \cdot \eta, \quad \forall \xi, \eta \in \bRn.
		$$
		Moreover, we also have that $\mathbf{W}_{\s}\geq \mathbf{V}^{*}_{\s}$ and $\mathbf{V}_{\s}\geq \omega^{*}_{\s}$.
	\end{prop}
	
	\begin{proof}
		Take $\s \in [\! [\tau: T-1 ]\!]$. It suffices to prove the inequality for $\xi\in {\rm dom}(\mathbf{V}_{\s})$ and $\eta\in{\rm dom}(\mathbf{W}_{\s})$, as the inequality holds trivially if either value is $+\infty$.  The proof does not assume $\mathbf{V}_\s$ or $\mathbf{W}_\s$ are proper; we will, in fact, deduce that they must be bounded from below.
		
		Let $\left\{x_t\right\}_{t=\mathrm{s}}^T \in \mathscr{N}_{\s}$ and $\left\{p_t\right\}_{t=\mathrm{s}}^T \in \mathscr{P}_{\s}$ be feasible trajectories for the optimization problems \eqref{primalvaluefunction} and \eqref{dualvaluefunction}, respectively. From the definition of conjugate, we have, for every $\omega\in \Omega$:
		\begin{align*} L_t\left(\omega,x_{t-1}, \Delta x_t\right)+M_t\left(\omega,(\espc{t}{p_t}), (\espc{t}{\Delta p_t})\right) &\geq x_{t-1} \cdot (\espc{t}{\Delta p_t})+\Delta x_t \cdot (\espc{t}{p_t}),\\
			&=x_t \cdot (\espc{t}{p_t})-x_{t-1} \cdot \espc{t}{p_{t-1}}.
		\end{align*} 
		Thus, we have that
		$$
		\esp{L_t\left(\omega,x_{t-1}, \Delta x_t\right)+M_t\left(\omega,(\espc{t}{p_t}), (\espc{t}{\Delta p_t})\right)}\geq \esp{x_t \cdot (\espc{t}{p_t})-x_{t-1} \cdot \espc{t}{p_{t-1}}}.
		$$
		Using the properties of conditional expectation, we obtain
		$$
		\esp{L_t\left(\omega,x_{t-1}, \Delta x_t\right)+M_t\left(\omega,(\espc{t}{p_t}), (\espc{t}{\Delta p_t})\right)}\geq \esp{x_t \cdot p_t-x_{t-1} \cdot p_{t-1}}.
		$$
		By the linearity of expectation and the definition of $\mathscr{P}_{\s}$, we have
		\begin{align*}
			\esp{\sum_{t=\s+1}^T L_t\left(\omega,x_{t-1}, \Delta x_t\right)+\sum_{t=\s+1}^T M_t\left(\omega,(\espc{t}{p_t}), (\espc{t}{\Delta p_t})\right)} &\geq \esp{x_T \cdot p_T-x_{\s} \cdot p_{\s}},\\
			&=p_T\cdot\esp{x_T}+\xi \cdot \eta.
		\end{align*}
		Since $f(b)=g^*(-b)$, the definition of the conjugate of $g$ allows us to deduce that, almost everywhere:
		$$
		g\left(\esp{x_T}\right)+f(p_{T})\geq -p_{T}\cdot\esp{x_T}.
		$$
		It follows that:
		$$
		\esp{\sum_{t=\s+1}^T L_t\left(\omega,x_{t-1}, \Delta x_t\right)}+g\left(\esp{x_T}\right)+ \esp{\sum_{t=\s+1}^T M_t\left(\omega,\espc{t}{p_t},\espc{t}{\Delta p_t}\right)}+f(p_{T})\geq \xi\cdot \eta.
		$$
		Taking the infimum over $\{x_{t}\}_{t=\s}^{T}\in \mathscr{N}_{\s}$ and $\{p_{t}\}_{t=\s}^{T}\in \mathscr{P}_{\s}$, we obtain the desired inequality. In particular, we deduce (a posteriori) that $\mathbf{V}(\xi)>-\infty$ and $\mathbf{W}_{\s}(\eta)>-\infty$.
	\end{proof}
	
	\subsection{Strong duality}
	
	Let us now focus on proving a strong duality result, which in our case means that the dual value function is the conjugate of the primal value function. For such purpose we introduce some notation first. For $x \in \bRn$, $t \in[\![\tau+1: T]\!]$ and $\omega \in \Omega$ we set
	\[
	\Gamma_L(t,\omega, x):=\left\{v \in \bRn \mid L_{t}(\omega,x,v) \in \mathbb{R}\right\}\]
	and 
	\[\mathbf{X}(t,\omega):=\left\{x \in \bRn \mid \Gamma_L(t,\omega, x) \neq \emptyset\right\}.\]
	\begin{rem}
		Notice that the minimization in  \eqref{primalvaluefunction} can be restricted to trajectories whose initial state $\esp{x_\s}=\xi$ is brought to the target $\dom(g)$ at time $t=T$. This means that the set $\dom(g)$ can be understood as a terminal constraint, implicitly  encoded in the formulation of the problem.  Similarly, by allowing each $L_t$ to take infinite values, we are handling implicitly constraints over the state of the system $x_t$ and the variation $\Delta x_t$. Indeed, any feasible trajectory of the Bolza problem  \eqref{primalvaluefunction} must satisfy, a.s.~on $\omega\in\Omega$:
		\begin{eqnarray*}
			x_{t-1}\in\X(t,\omega)\quad\text{and}\quad\Delta x_t\in \Gamma_{L}(t,\omega,x_{t-1}),\quad\forall t\in[\![\tau+1:T]\!].
		\end{eqnarray*}
		In other words, the set-valued maps $\Gamma_{L}$ and $\X$ correspond respectively to the dynamics of the system and  to the (time-dependent) state constraint.
	\end{rem}
	
	The qualification conditions we use in this paper are the following:
	
	\begin{assumption} \label{h1} The primal problem \eqref{primalvaluefunction} satisfies the bounded recourse condition, i.e., for any $t\in [\! [\tau+1: T]\!]$ it is true that:
		\begin{itemize}
			\item [a)] For any $\rho>0$ there is a summable function $\beta: \Omega \rightarrow \mathbb{R}$ such that, a.s.~on $\omega \in \Omega$ we have      \[L_t\left(\omega, x, v\right) \leq\beta\text{ whenever }x \in \mathbf{X}(t,\omega) \text { and }\Gamma_L\left(t,\omega, x\right) \text { and }|x|,|v| \leq \rho\]
			
			\item [b)] For every $\rho>0$ there is a $\rho^{\prime}>0$ such that  a.s.~on $\omega \in \Omega$, we have that if $x \in \mathbf{X}(t,\omega)$ with $\left|x\right| \leq \rho$, then, setting set $\mathbf{X}(T+1,\omega)=\bRn$, 
			\[\exists v \in \Gamma_L\left(t,\omega, x\right) \text { such that }  x+v \in \mathbf{X}(t+1,\omega) \text { and }\left|x+v\right| \leq \rho^{\prime}.
			\]
		\end{itemize}
	\end{assumption}
	
	\begin{assumption} \label{h2} The primal problem \eqref{primalvaluefunction} satisfies  the strict feasibility condition, i.e., there exists some $\bar{x} \in \mathscr{N}_{\tau}, \varepsilon>0$ and summable functions $\alpha_{\tau+1},\ldots,\alpha_{T}$, such that $\esp{\bar{x}_T} \in \operatorname{dom}(g)$, and for $t\in [\! [\tau+1: T]\!]$ a.s.~on $\omega \in \Omega$, it also holds that
		\[x \in \mathbf{X}(t,\omega), \quad v \in \Gamma_L\left(t,\omega, x\right) \quad \text { and } \quad L_t\left(\omega, x, v\right) \leq \alpha_t\]
		whenever
		$\left|x-\bar{x}_{t-1}\right| \leq \varepsilon$ and $\left|v-\Delta \bar{x}_t\right| \leq \varepsilon$.
	\end{assumption}
	
	\begin{assumption} \label{h3} The set-valued map $\omega \rightarrow \mathbf{X}(t,\omega)$ is $\mathscr{G}_{t-1}-$measurable for each $t\in [\![\tau+1:T]\!]$
	\end{assumption}

	We are now in a position to establish a strong duality relation between the primal and dual value functions.
	
	\begin{prop}\label{prop:conjugacy_1}
		Under Assumptions \ref{h1}-\ref{h3} and assuming that \eqref{primalvaluefunction} attains its infimum, it follows that $\mathbf{V}^{*}_{\s}=\mathbf{W}_{\s}$ for each $\mathrm{s} \in[\![\tau: T-1]\!]$. Furthermore, the infimum in the definition of $\mathbf{W}_{\s}(\eta)$ is attained for any $\eta\in \operatorname{dom}(\mathbf{W}_{\s})$.
	\end{prop}
	\begin{proof}
		Let us take $\mathrm{s} \in[\![\tau: T-1]\!]$. Note that from Assumptions \ref{h1}-\ref{h3}, we have $\mathbf{V}_{\s}(\esp{\bar{x}_{\s}})<+\infty$. In particular, it follows that:
		$$
		\mathbf{V}^{*}_{\s}(\eta)\geq \esp{\bar{x}_{\s}}\cdot \eta -\mathbf{V}_{\s}(\esp{\bar{x}_{\s}})>-\infty,\quad \forall \eta \in \bRn.
		$$
		First, we prove that $\mathbf{W}_{\s}=\mathbf{V}^{*}_{\s}$.  Note that for $\eta \in \operatorname{dom}(\mathbf{V}^{*}_{\s})$, we have that
		$$
		-\mathbf{V}_{\s}^*(\eta)=\inf _{\xi \in \bRn}\left\{\mathbf{V}_{\s}(\xi)-\xi \cdot \eta\right\}=\inf _{\mathbf{x} \in \mathscr{N}_{\s}}\left\{\esp{\sum_{t=\s+1}^T L_t\left(\omega,x_{t-1}, \Delta x_t\right)}+g\left(\esp{x_T}\right)-\esp{x_{\s}} \cdot \eta\right\},
		$$
		where we have considered $\mathbf{x}=\{x_{t}\}_{t=\s}^{T}$. Define $\ell(a,b)=g(b)-a\cdot \eta$.  It is not difficult to see that $\operatorname{dom}(\ell)=\bRn\times \operatorname{dom}(g)$. Therefore, by virtue of \cite[Thm.5]{RocWet83}, it follows that Assumptions \ref{h1}-\ref{h3} imply the existence of $p=(p_{\s},\ldots,p_{T})\in \mathscr{L}^1_{\s}$ such that $p\in \partial \Phi(0)$, where for any $\mathbf{y}=\{y_{t}\}_{t=\s}^{T}\in \mathscr{L}^{\infty}_{\s}$, the perturbation function $\Phi$ is given by:
		\begin{equation}\label{eq:perturbation}
			\Phi(\mathbf{y})=\inf_{x\in \mathscr{N}_{\s}}\left\{ \esp{\sum_{t=\s+1}^T L_t\left(\omega,x_{t-1}-\espc{t}{\Delta y_{t-1}}, \Delta x_t+\mathbb{E}^{t}_{\Delta}\left[y_{t-1}\right]+\espc{t}{y_{t}}\right)}+ \ell\left(\esp{x_{\s}+y_{\s}},\esp{x_{T}}\right)\right\}.  
		\end{equation}
		Here $\mathbb{E}^{t}_{\Delta}[z]:=\espc{t}{z}-\espc{t-1}{z}$. Note that the assumption that \eqref{primalvaluefunction} attaining its infimum guarantees that $p$ is in $\mathscr{L}^1_{\s}$, which is needed in \cite[Thm.6]{RocWet83}. 
		From \cite[Thm.6]{RocWet83}, we also obtain that $\mathbf{V}_\s^*(\eta)$ is the optimal value of the problem:
		\begin{equation}\label{prob:aux1}
			\min_{\mathbf{p}\in \mathscr{P}_{\s}}\left\{\esp{\sum_{t=\s+1}^{T}M_{t}(\omega,(\espc{t}{p_{t}}),(\espc{t}{\Delta p_{t}})}+\ell^{*}(\espc{\s}{p_{\s}},-p_{T})\mid \mathbf{p}=\{p_{t}\}_{t=\s}^{T}\right\},
		\end{equation}
		where the minimum is attained by $p=(p_{\s},\ldots,p_{T})\in \mathscr{P}_{\s}\cap\partial\Phi(0)$.  To show this equals $\mathbf{W}_\s(\eta)$, we focus on the $\ell^*$ term.  By definition, $\ell(a,b)=g(b)-a\cdot\eta$, so its conjugate $\ell^{*}(a,b)$ is finite only if $a=-\eta$, in which case $\ell^*(a,b)=g^{*}(b)$.  Combining this definition with the definition of the terminal cost  $f(p_T)=g^*(-p_T)$ in problem~\eqref{prob:aux1}, we recover the definition of $\mathbf{W}_s(\eta)$ in \eqref{dualvaluefunction}.  We thus conclude that $\mathbf{W}_\s=\mathbf{V}^*_\s$.  The equality also holds for $\eta\not\in{\rm dom}(\mathbf{V}^*)$ by Proposition~\ref{prop:weak}. 
	\end{proof}
	
	\section{Discrete-time method of characteristics}
	We now present and prove a discrete-time method of characteristics, which describes the evolution of the subgradients of the primal value function by means of a discrete-time Hamiltonian system.
	
	Let $H_{\tau+1},\ldots,H_T:\Omega\times\bRn\times\bRn\to\bR\cup\{\pm\infty\}$ be the Hamiltonians associated with the primal problem \eqref{primalvaluefunction}, that is
	\[H_t(\omega,x,p):=\sup_{v\in\bRn}\left\{p\cdot v-L_t(\omega,x,v)\right\},\,~\forall \omega\in \Omega \mbox{ and } x,p\in\bRn,\ \forall t\in[\![\tau+1:T]\!].\]
	For each $t\in[\![\tau+1:T]\!]$ and $\omega\in \Omega$ fixed, $H_t(\omega,\cdot,\cdot)$ is a concave-convex function from $\bRn\times \bRn$ to $\mathbb{R}\cup\{\pm\infty\}$, thus its sub-differential is given by \eqref{eq:subdif_saddle}. In particular, we have that $(-w,v)\in \partial H_t(\omega,\bar x,\bar p)$ if and only if
	\[H_t(\omega,x,\bar p)+w\cdot(x-\bar x)\leq H_t(\omega,\bar x,\bar p)\leq H_t(\omega,\bar x,p)-v\cdot(p-\bar p),\qquad\forall x,p\in\bRn.\]
	\begin{dfn}\label{dfn:Hamiltonian}
		Let $\s\in [\![\tau:T-1]\!]$ be given. We say that $\{(x_t,p_t)\}_{t=\s}^T\subset\mathscr{N}_{\tau}\times\mathscr{P}_{\tau}$ is a \emph{discrete-time Hamiltonian trajectory} on $[\![\s:T]\!]$ provided that, a.s.~on $\omega \in \Omega$
		\begin{equation}\label{eq:dHt}
			(-\espc{t}{\Delta p_t},\Delta x_t)\in\partial H_t(\omega,x_{t-1},\espc{t}{p_t}),\qquad\forall t\in[\![\s+1,T]\!].
		\end{equation}
	\end{dfn}
	
	The method of characteristics, the main result of this paper, is outcome of the following two results. The first one, says that the transversality condition transported backward on time through a discrete-time Hamiltonian trajectory, leads to a subgradient of the value function. Observe that for the following result no qualification conditions are needed.
	
	\begin{teo}\label{thm:charact_thrm_1}
		Let  $\xi,\eta\in\bRn$ be given. Suppose that there exists a  \emph{discrete-time Hamiltonian trajectory} on $[\![\tau:T]\!]$, say  $\{(x_t,p_t)\}_{t=\tau}^T\subset\mathscr{N}_{\tau}\times\mathscr{P}_{\tau}$, such that $(\esp{x_\tau},\esp{p_\tau})=(\xi,-\eta)$ and that satisfies the transversality condition
		\begin{equation}\label{eq:transversality}
			-p_T\in\partial g\left(\esp{x_T}\right).
		\end{equation}
		Then, $\eta\in\partial\mathbf{V}_\tau(\xi)$ and $-\esp{p_\s}\in\partial \mathbf{V}_s(\esp{x_\s})$ for any $\s\in [\![\tau:T-1]\!]$.
	\end{teo}
	\begin{proof}
		Define the function $\ell:\bRn\times\bRn\to\bR\cup\{\pm\infty\}$ as
		\[\ell(a,b):=g(b)-a\cdot\eta,\qquad\forall a,b\in\bRn.\]
		Note that it is a convex, proper, and l.s.c function.  
		Let  $\{(x_t,p_t)\}_{t=\tau}^T\subset\mathscr{N}_{\tau}\times\mathscr{P}_{\tau}$ be a Hamiltonian trajectory satisfying the hypotheses on the statement.  Then, by condition \eqref{eq:transversality} and $\esp{p_\tau}=-\eta$, we have that
		\begin{equation*}
			(\esp{p_\tau},-p_T)\in\partial\ell(\esp{x_\tau},\esp{x_T}).	
		\end{equation*}
		As $(p_{\tau},\ldots,p_{T})\in \mathscr{P}_{\tau}$, $\espc{\tau}{p_\tau}$ and $p_{T}$ are constants, which yields:
		\begin{equation}\label{eq:tc}
			(\espc{\tau}{p_\tau},-p_T)\in\partial\ell(\esp{x_\tau},\esp{x_T}).	
		\end{equation}
		Furthermore, since $H_t(\omega,x,p)=(L_t(\omega,x,\cdot))^*(p)$, by virtue of \cite[Theorem 37.5]{RocBook70}, the Hamiltonian inclusion \eqref{eq:dHt} is equivalent (a.s.~on $\omega$) to the \emph{discrete-time Euler-Lagrange relation} given by
		\begin{equation}\label{eq:dEUr}
			(\espc{t}{\Delta p_t},\espc{t}{p_t})\in\partial L_t(\omega,x_{t-1},\Delta x_t),\qquad\forall t\in [\![\tau+1:T]\!].
		\end{equation}
		According to \cite[Theorem 4]{RocWet83}, these are the optimality conditions for the auxiliary problem 
		\begin{equation}\label{eq:auxp}\tag{$\mathcal{P}_0$}
			\inf_{\mathbf{y}\in\mathscr{N}_{\tau}}\left\{\esp{\sum_{t=\tau+1}^TL_t(\omega, y_{t-1},\Delta y_t)}+g\left(\esp{y_T}\right)-\esp{y_\tau}\cdot\eta\right\}=\inf_{y\in\bRn}\left\{\mathbf{V}_\tau(y)-y\cdot\eta\right\}.
		\end{equation}
		Consequently, $\{(x_t,p_t)\}_{t=\tau}^T$ is an optimal solution for \eqref{eq:auxp}, and $\{p_t\}_{t=\tau}^T$ is an associated dual solution in $\partial\Phi(0)\cap \mathscr{L}^1_{\s}$ (where $\Phi$ is the perturbation function from \eqref{eq:perturbation}). By \cite[Theorem 6]{RocWet83}, it then follows that $\{p_t\}_{t=\tau}^T\subset \mathscr{L}^1_{\s}$ is an optimal trajectory realizing $\mathbf{W}_\tau(\eta)$, and that strong duality holds, i.e., $\mathop{\rm val}\eqref{eq:auxp}=-\mathbf{W}_\tau(\eta)$.
		Since  $\{x_t\}_{t=\tau}^T$ is an optimal solution for \eqref{eq:auxp} with $E x_\tau=\xi$, we have $\mathop{\rm val}\eqref{eq:auxp} = \mathbf{V}_\tau(\xi) -\xi\cdot\eta$. Combining this with the strong duality result $\mathop{\rm val}\eqref{eq:auxp}=-\mathbf{W}_\tau(\eta)$, we obtain the Fenchel-Young equality $\mathbf{V}_\tau(\xi) + \mathbf{W}_\tau(\eta) = \xi \cdot \eta$.  This implies $\eta\in\partial\mathbf{V}_\tau(\xi)$.
		
		The result extends to any starting time $\s\in [\![\tau:T-1]\!]$ because the restricted trajectory $\{(x_t,p_t)\}_{t=\s}^T$ is itself a Hamiltonian trajectory for the problem on $[\![\s:T]\!]$, allowing the same proof to be applied.
	\end{proof}
	
	The second result is the converse of the Theorem~\ref{thm:charact_thrm_1}. It provides a complete characterization of the subdifferential of the value function, since it says that for any subgradient of the value function there is a discrete-time Hamiltonian trajectory that  allows to recover that subgradient.
	\begin{teo}\label{thm:charact_thrm_2}
		
		Let $\xi \in \bRn$ and $\eta \in \partial \mathbf{V}_{\tau}(\xi)$. Suppose $x\in \mathscr{N}_{\tau}$ is a minimizer of the problem \eqref{primalvaluefunction}.  Then, if the hypotheses \ref{h1}-\ref{h3} are in force, the converse of the Theorem \ref{thm:charact_thrm_1} holds, i.e., there exists a discrete-time Hamiltonian trajectory $\{(x_t,p_t)\}_{t=\tau}^T \in \mathscr{N}_{\tau} \times \mathscr{P}_{\tau}$ that satisfies $(\esp{x_\tau},\esp{p_\tau})=(\xi,-\eta)$ and the transversality condition $-p_T\in\partial g\left(\esp{x_T}\right)$ is satisfied.
	\end{teo}
	
	\begin{proof}
		Let $\eta \in \partial \mathbf{V}_{\tau}(\xi)$.  This condition implies that $\xi \in \partial \mathbf{W}_{\tau}(\eta)$, as $\mathbf{V}_{\tau}^{*}(\eta)=\mathbf{W}_{\tau}(\eta)$. Hence, $\eta \in \operatorname{dom}(\mathbf{W}_{\tau})$, and Proposition~\ref{prop:conjugacy_1} applies, providing the existence of an optimal trajectory $\{p_{t}\}_{t=\tau}^{T}\in \mathscr{P}_{\tau}$ such that $\esp{p_{\tau}}=-\eta$, and
		$$
		\mathbf{W}_{\tau}(\eta)=\esp{\sum_{t=\tau+1}^{T}M_{t}(\omega,\espc{t}{p_{t}},\espc{t}{\Delta p_{t}})} + f(p_{T}).
		$$
		From the Fenchel-Young equality we have that $\mathbf{V}_{\tau}(\xi)+\mathbf{W}_{\tau}(\eta)=\xi\cdot \eta$. Thus $\{x_{t}\}_{t=\tau}^{T}$ is an optimal solution for \eqref{eq:auxp}. 
		
		We have an optimal trajectory $\{x_t\}_{t=\tau}^T$ and an optimal dual trajectory $\{p_t\}_{t=\tau}^T$ (which, as shown in the proof of Prop.~\ref{prop:conjugacy_1}, satisfies $\{p_{t}\}_{t=\tau}^{T}\in \partial\Phi(0)$).  By \cite[Thm.4]{RocWet83}, this primal-dual pair must satisfy the discrete-time Euler-Lagrange relation \eqref{eq:dEUr} and the transversality condition \eqref{eq:tc}.  As established in the proof of Theorem~\ref{thm:charact_thrm_1}, the Euler-Lagrange relation is equivalent to the Hamiltonian inclusion \eqref{eq:dHt}.  Therefore, $\{(x_t,p_t)\}_{t=\tau}^T$ is a discrete-time Hamiltonian trajectory satisfying the required transversality condition.
	\end{proof}
	
	\section{Applications to linear-convex problems}\label{sec:L-Cproblems}
	
	In this section, we demonstrate how our main theory applies to the important class of linear-convex (LC) optimal control problems. Our objective is to show that the LC problem can be reformulated as a convex stochastic Bolza problem of the form \eqref{primalvaluefunction}, and establish a version of the method of characteristics for this family.
	
	Let $\ell_{t}:\bRn\times \mathbb{R}^{m}$ be convex functions for each $t=\tau,\ldots,T-1$.  We consider a disturbance process (or noise) $\{w_{t}\}_{t=\tau}^{T-1}$, where the $w_t$ are i.i.d. with $\esp{w_t}$ = 0, $\esp{w_tw_t^\top}=W$.  Let $\gamma$ be a random variable, independent of all $w_{t}$, which we use to define the initial information $w_{\tau-1}=\gamma$.  Furthermore, let $U_{t}\subset \mathbb{R}^{m}$ and $X_{t}\subset \bRn$ be closed convex sets, with the initial state set $X_{\tau}$ also assumed to be bounded.  Given an initial expected state $\xi \in \bRn$ and $\tau\leq T-1$, we consider the state process $x=(x_{\tau},x_{\tau+1},\ldots,x_{T})$ and the control process $u=(u_{\tau},\ldots,u_{T-1})$. The LC problem with mixed constraints is formulated as:
	\begin{equation}\tag{LC}\label{example:LCMC}
		\begin{cases}
			\text{ Minimize }&\esp{ \displaystyle\sum_{t=\tau}^{T-1}\ell_t(x_t,u_t)}+g\left(\esp{x_T}\right),\\
			\text{ over all }& \{x_t\}_{t=\tau}^T\in \mathscr{N}_{\tau}, ~\{u_{t}\}_{t=\tau}^{T-1}\in \mathscr{M}_{\tau},\\ 
			\text{ subject to } & x_{t+1}=Ax_{t}+Bu_{t}+w_{t}\\
			&f_{t}(x_{t},u_{t}) \leq 0 \quad \text{a.s. } \forall t\in[\![\tau:T-1]\!],\\
			&u_{t} \in U_{t}\quad \text{a.s. }\forall t\in[\![\tau:T-1]\!],\\
			&x_{t} \in X_{t}\quad \text{a.s. }\forall t\in[\![\tau:T]\!],\\
			&\esp{x_{\tau}}=\xi, 
		\end{cases}
	\end{equation}
	where the sets $\mathscr{N}_{\tau},\mathscr{M}_{\tau}$ are given by:
	\begin{align*}
		\mathscr{N}_{\tau}&:=\{(x_{\tau},\ldots,x_{T})\in \mathscr{L}^{\infty}(\Omega,\mathscr{A},\mu;\bRn)^{T-\tau+1}\mid x_{t} \text{ is } \mathscr{G}_{t}\text{-measurable for } t \in [\![\tau:T]\!]\},\\
		\mathscr{M}_{\tau}&:=\{(u_{\tau},\ldots,u_{T-1}) \in \mathscr{L}^{\infty}(\Omega,\mathscr{A},\mu;\mathbb{R}^{m})^{T-\tau} \mid u_{t}\text{ is } \mathscr{G}_{t}\text{-measurable for } t \in [\![\tau:T-1]\!]\},
	\end{align*}
	where $w_{\tau-1}=\gamma$, and $\mathscr{G}_t := \sigma(w_{\tau-1},\ldots,w_{t-1})$ is the $\sigma$-field generated by the noise up to time $t-1$. This filtration $\mathscr{G}_t$ represents the information available when the control $u_t$ is chosen. The full $\sigma$-field is $\mathscr{A}=\mathscr{G}_T = \sigma(w_{\tau-1},\ldots, w_{T-1})$.
	
	Our strategy is to first reformulate Problem~\eqref{example:LCMC} as an equivalent stochastic Bolza problem (stated in Lemma~\ref{lem:LC2Bolza}). We then apply the main theory of this paper (Theorem~\ref{thm:charact_thrm_1} and Lemma~\ref{lem:Char2Bolza}) to this Bolza formulation to obtain the main result of this section. This two-step process is illustrated in Figure~\ref{fig:lc_roadmap}.
	
	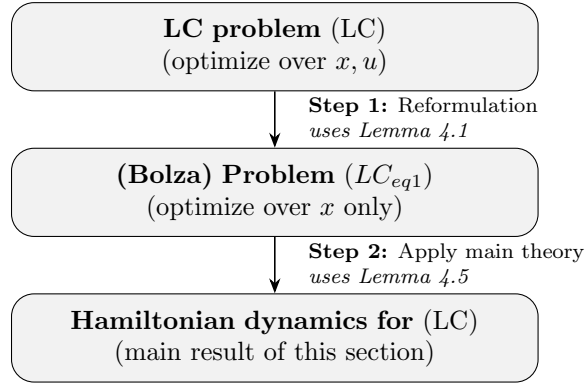
\begin{figure}[h!]
		\centering
		\begin{tikzpicture}[
			>=Stealth,
			node distance=0.75cm,
			box/.style={
				draw,
				rounded corners=3mm,
				fill=gray!10,
				text width=6.2cm,
				minimum height=0.95cm,
				inner xsep=4mm,
				inner ysep=2mm,
				align=center,
				font=\footnotesize
			},
			arr/.style={-Stealth, semithick},
			lab/.style={font=\scriptsize, align=left}
			]
			
			\node (lc) [box] {\textbf{LC problem \eqref{example:LCMC}}\\(optimize over $x,u$)};
			\node (bolza) [box, below=of lc] {\textbf{(Bolza) Problem \eqref{example:LCMCeq3}}\\(optimize over $x$ only)};
			\node (ham) [box, below=of bolza] {\textbf{Hamiltonian dynamics for \eqref{example:LCMC}}\\(main result of this section)};
			
			\draw[arr] (lc) -- (bolza)
			node[lab, right=3mm, midway] {\textbf{Step 1:} Reformulation\\\emph{uses Lemma~\ref{lem:LC2Bolza}}};
			\draw[arr] (bolza) -- (ham)
			node[lab, right=3mm, midway] {\textbf{Step 2:} Apply main theory\\\emph{uses Lemma~\ref{lem:Char2Bolza}}};
			
		\end{tikzpicture}
		\caption{Roadmap for the reformulation of the \eqref{example:LCMC} problem, and the resulting method of characteristics.}
		\label{fig:lc_roadmap}
	\end{figure}
	
	The LC problem involves optimizing over both the state $x$ and the control $u$. Our strategy is to "minimize out" the control variable $u$ to obtain a new problem that depends only on the state $x$.  We achieve this by absorbing all constraints and the control-dependent costs into a single function, $\Phi_t$, and then defining a new Bolza Lagrangian, $L_t$, as the infimum of $\Phi_t$ with respect to $u$.
	
	For every $t \in [\![\tau+1:T]\!]$, let us define the function $\Phi_{t}$ by:
	$$\Phi_{t}(\omega,x,v,u)=\ell_{t-1}(x,u) + \delta_{C_{t-1}}(x,v,u) + \delta_{D_{t-1}}(x,u) + \delta_{U_{t-1}}(u)+\delta_{X_{t-1}}(x),$$
	where $v = \Delta x_t$ and the constraint sets are:
	\begin{align*}
		C_{t-1}&:=\{(x,v,u)\mid v=Ax-x+Bu+w_{t-1}\}\\
		D_{t-1}&:=\{(x,u)\mid f_{t-1}(x,u)\leq 0 \}.
	\end{align*}
	The original LC problem \eqref{example:LCMC} is then equivalent to finding $\mathbf{V}_{\tau}(\xi)$, the value function associated to the following problem:
	
	\begin{equation}\tag{$LC_{eq}$}\label{example:LCMCeq2}
		\begin{split}
			\min &\, \esp{\sum_{t=\tau+1}^{T}\Phi_{t}(\omega,x_{t-1},\Delta x_{t},u_{t-1})}+g\left(\esp{x_T}\right),\\
			\mbox{over all } &\{x_t\}_{t=\tau}^T\in \mathscr{N}_{\tau},~\{u_{t}\}_{t=\tau}^{T-1}\in \mathscr{M}_{\tau}\text{ such that } \esp{x_{\tau}}=\xi 
		\end{split}
	\end{equation}
	
	We can rewrite this value function by separating the infima:
	$$\mathbf{V}_{\tau}(\xi)=\inf_{\substack{x\in\mathscr{N}_{\tau}\\\esp{x_\tau}=\xi}}\inf_{u\in \mathscr{M}_{\tau}}\left\{\esp{\sum_{t=\tau+1}^{T}\Phi_{t}(\omega,x_{t-1},\Delta x_{t},u_{t-1})}+g\left(\esp{x_T}\right)\right\}$$
	
	Our goal is to interchange the expectation and the infimum over $u$. If this is valid, we can define a new Lagrangian $L_t$ that only depends on $x$: 
	
	\begin{equation}\tag{9}\label{lagrangiandef:LCMCinfcase}
		L_{t}(\omega,x,v) := \inf_{u \in \mathbb{R}^{m}} \Phi_{t}(\omega,x,v,u)
	\end{equation}
	
	This would reduce the problem to the desired stochastic Bolza form: \begin{equation}\tag{$LC_{eq1}$}\label{example:LCMCeq3}
		\mathbf{V}_{\tau}(\xi)= \inf_{\substack{x\in\mathscr{N}_{\tau}\\\esp{x_\tau}=\xi}}\left\{\esp{\sum_{t=\tau+1}^{T}L_{t}(\omega,x_{t-1},\Delta x_{t})}+g\left(\esp{x_T}\right)\right\}
	\end{equation}
	
	\begin{lemma}{(Equivalence of Problems)}\label{lem:LC2Bolza}
		For any fixed $x\in \mathscr{N}_{\tau}$ with $\esp{x_{\tau}}=\xi$, and $t\in[\![\tau+1:T]\!]$, the function $\varphi_{t}(\omega,u) := \Phi_{t}(\omega,x_{t-1},\Delta x_{t},u)$ is a normal integrand. Consequently, the interchange of infimum and expectation is justified, and the value function $\mathbf{V}_\tau(\xi)$ of the \eqref{example:LCMC}~problem is equal to that of the stochastic Bolza problem~\eqref{example:LCMCeq3}
	\end{lemma}

	\begin{proof}
		Let $t\in[\![\tau+1:T]\!]$, and let $\omega \in \Omega$.  Assume that $\operatorname{dom}(\varphi_{t}(\omega,\cdot))\neq \emptyset$, as lower semicontinuity is trivial otherwise. The function $\varphi_t(\omega,\cdot)$ is l.s.c. as it is the sum of several components, each being l.s.c.: $u \rightarrow \ell_{t-1}(x,u)$ is convex and finite-valued for any fixed $x$, thus l.s.c. for any fixed $x$, which implies that  $u \rightarrow \ell_{t-1}(x_{t-1},u)$ is also l.s.c.; the sets $D_{t-1}$, $C_{t-1}$ and $U_{t-1}$ are closed by definition, so their indicator functions ($\delta_{D_{t-1}}$, $\delta_{C_{t-1}}$, and $\delta_{U_{t-1}}$) are also l.s.c.  We conclude by algebra of l.s.c. that $\varphi_t(\omega,\cdot)$ is l.s.c.
		
		We now focus on measurability of the set-valued map $\omega\mapsto \operatorname{epi}(\varphi_t(\omega,\cdot))$. Note that, for each $\omega$ its epipragh is the intersection of three set-valued maps:
		$$
		\operatorname{epi}(\varphi_{t}(\omega,\cdot)) = \underbrace{\operatorname{epi}(\ell_{t-1}(x_{t-1},\cdot))}_{(i)} \cap \underbrace{(U_{t-1} \times \mathbb{R})}_{(ii)} \cap \underbrace{(V_{t} \times \mathbb{R})}_{(iii)},
		$$
		where, for each $\omega$, the set $V_{t}$ is defined by:
		\begin{align*}
			V_{t} &= \left\{ u \in \mathbb{R}^{m} \mid \Delta x_{t} = A x_{t-1} - x_{t-1} + B u + w_{t-1}\,\mbox{and}\, f_{t-1}(x_{t-1},u)\leq 0\right\}.
		\end{align*}
		
		Map $(ii)$ is constant and, thus, measurable.  We need to prove the measurability of maps $(i)$ and $(iii)$ with respect to $\sigma(\mathbf{W}_{\tau-1},\ldots,\mathbf{W}_{t-2})$:
		\begin{itemize}
			\item[$(i)$] Since $x \in \mathscr{N}_{\tau}$, we have in particular that $x_{t-1} \in \sigma(w_{\tau-1}, \ldots, w_{t-2})$, which, along with the continuity of the function $(x,u) \rightarrow \ell_{t-1}(x,u)$, allows us to conclude that the function $(\omega,u) \rightarrow \ell_{t-1}(x_{t-1},u)$ is a Carathéodory function. In particular, it is a normal integrand, which implies that its epigraphical map $\omega\mapsto \operatorname{epi}(\ell_{t-1}(x_{t-1},\cdot))$ is $\sigma(w_{\tau-1}, \ldots, w_{t-1})$-measurable.
			\item[$(iii)$] We can write $E_t=\left\{u\,\mid\,h_t(\omega,u)=0\,\mbox{and}\,m_t(\omega,u)\leq 0\right\}$, where
			\begin{align*}
				h_t(\omega,u)&:= A x_{t-1} - x_{t-1} + B u + w_{t-1} - \Delta x_{t}\\
				m_t(\omega,u)&:= f_{t-1}(x_{t-1},u)
			\end{align*}
			Both, $h_t$ and $m_t$ are Carathéodory functions due to their continuity in $u$-variable (by assumption on $f_{t-1}$ and measurable in $\omega$ (since $x_{t-1}$, $\Delta x_t$, and $\mathbf{W}_{t-1}$ are all $\sigma(\mathbf{W}_{\tau-1},\ldots,\mathbf{W}_{t-1})$-measurable.  By \cite[Thm.14.36]{RockafellarWets1998}, a set-valued map defined by level sets of Carathéodory functions is measurable.  Thus, $\omega\mapsto E_t$ is measurable. 
		\end{itemize}
		
		Since the maps $(i)$, $(ii)$, and $(iii)$ are all measurable, their intersection $\omega\mapsto \operatorname{epi}(\varphi_t(\omega,\cdot))$ is also measurable.  This property, combined with the l.s.c. established previously, provides that $\varphi_t$ is a normal integrand. The interchangeability between infimum and expectation follows by virtue of \cite[Thm.14.60]{RockafellarWets1998}.
	\end{proof}
	
	A remaining difficulty is that an optimal solution $\bar{x}$ for the Bolza problem \eqref{example:LCMCeq3} may not correspond to an optimal pair $(\bar{x}, \bar{u})$ for the original problem. We need to ensure that the infimum in the definition of $L_t$ \eqref{lagrangiandef:LCMCinfcase} is attained by some measurable $\bar{u}$. We also need to ensure the Bolza problem \eqref{example:LCMCeq3} itself has a solution. The following assumptions provide sufficient conditions for both.
	
	\begin{assumption} \label{hyp:USCcondition}
		For all $t\in [\![\tau+1:T-1]\!]$, there exists an upper semicontinuous function $\psi_t: \bRn \rightarrow \mathbb{R}$ such that $\sup _{u \in \mathbf{U}}\left\{|u| \mid f_t(x, u) \leq 0\right\} \leq \psi_t(x)$, for all $x$ in $\bRn$.
	\end{assumption}
	
	\begin{assumption} \label{assumptionAL1:LCMC} There exists a $x\in \mathscr{N}_{\tau}$ such that $L_{t}(\omega,x_{t-1},\Delta x_{t})<+\infty$, $\forall t \in [\![\tau+1:T-1]\!]$ a.s., and $\esp{x_{\tau}}=\xi$
	\end{assumption}
	
	\begin{assumption} \label{assumptionBL1:LCMC} There exists $c_{1},c_{2}\in \mathbb{R}$ and a proper univariate non-decreasing coercive single-valued function  $\theta$\footnote{Meaning that $\theta$ is bounded from below with $\lim_{s\rightarrow+\infty}\frac{\theta(s)}{s}=+\infty$}  such that for every $\omega \in \Omega$: $L_{t}(\omega,x,v)\geq \theta(\max\{0,|v|-c_{1}|x|\})-c_{2}|x|$, for all $(x,v)\in \bRn\times\bRn$.
	\end{assumption}

	\begin{lemma}{(Existence of Optimal Solutions)}\label{lem:Char2Bolza}
		Let assumptions \ref{hyp:USCcondition}, \ref{assumptionAL1:LCMC}, and \ref{assumptionBL1:LCMC} be satisfied. Then:
		\begin{enumerate}
			\item\label{lem:c2b1} (Properties of $L_t$) The infimum in the definition of $L_t(\omega, x, v)$ is attained. Furthermore, $L_t$ is a convex, proper, and l.s.c. function.
			\item\label{lem:c2b2} (Existence for Bolza) The stochastic Bolza problem~\eqref{example:LCMCeq3} admits an optimal solution $\bar{x} \in \mathscr{N}_\tau$.
			\item\label{lem:c2b3} (Existence for~\eqref{example:LCMC} There exists a corresponding $\bar{u} \in \mathscr{M}_\tau$ (given by a measurable selection) such that the pair $(\bar{x}, \bar{u})$ is an optimal solution to the original \eqref{example:LCMC}~problem
		\end{enumerate}    
	\end{lemma}
	
	\begin{proof}
		We prove each of the three claims in order
		
		\textit{1. Properties of $L_t$}: For a fixed $\omega \in \Omega$, the function $L_t(\omega, \cdot, \cdot)$ is the partial infimum of the convex, l.s.c. function $\Phi_t$, which is itself l.s.c. as a sum of l.s.c. functions.  Assumption~\eqref{hyp:USCcondition} provides a compact set for the control variable $u$ for any feasible $x$. Since we are minimizing a l.s.c. function ($\Phi_t$) over a compact set, the infimum in the definition of $L_t(\omega, x, v)$ is guaranteed to be attained whenever $L_t < +\infty$.Furthermore, the operations of partial minimization and infimal projection of a convex, proper, l.s.c. function (over a convex set) preserve these properties. Therefore, $L_t(\omega, \cdot, \cdot)$ is also convex, proper, and l.s.c. (see, e.g., \cite[Proposition 4.4]{deride2024subgradientevolution}).
		
		\textit{2. Existence for Bolza}: We prove the existence of an optimal solution $\bar{x}\in\mathscr{N}_\tau$ for the Bolza problem~\eqref{example:LCMCeq3} by the direct method in the calculus of variations.  We first relax the problem to a larger space $L^{1}(\Omega,\mathscr{A},\mu,\mathbb{R}^{T+1-\tau}\times \bRn)$ and, consequently, define the problem on the space of $L^1$-adapted processes $\mathscr{N}_\tau^1$.  Then, we prove the existence of an optimal solution $\bar{x}$ in this space by showing the associated objective functional is coercive and weakly l.s.c.  The proof is completed by showing that this solution actually is essentially bounded, and thus, $\bar{x}\in\mathscr{N}_\tau$, a solution to the original problem.
		
		Consider the infimum of the problem over the space of $L^1$-adapted processes $\mathscr{N}_\tau^1$:
		$$ \mathscr{N}^{1}_{\tau}:=\{x=(x_{\tau},\ldots,x_{T})\in L^{1}(\Omega,\mathscr{A},\mu,\mathbb{R}^{T+1-\tau}\times \bRn)\,\mid\, x_{t+1} \mbox{ is } \mathscr{G}_t-\mbox{measurable}\},
		$$
		defined by the value function
		\begin{equation}\tag{$LC_{L^{1}}$}\label{example:LCMCL1}
			\operatorname{val}(P_{\xi}^{\tau}):=\inf_{\substack{x_{t}\in \mathscr{N}^{1}_{\tau}\\\esp{x_{\tau}}=\xi}}\left\{\esp{\sum_{t=\tau+1}^{T}L_{t}(\omega,x_{t-1},\Delta x_{t})}+g\left(\esp{x_T}\right)\right\}.
		\end{equation}
		Define the objective functional $f_{obj}:L^{1}(\Omega,\mathscr{A},\mu,\mathbb{R}^{T+1}\times \bRn)\rightarrow \overline{\mathbb{R}}$ as:
		$$
		f_{obj}(x)=\esp{\sum_{t=\tau+1}^{T}L_{t}(\omega,x_{t-1},\Delta x_{t})} + g\left(\esp{x_T}\right)+\delta_{\{\xi\}}(\esp{x_{\tau}}) + \delta_{\mathscr{N}^{1}_{\tau}}(x).
		$$
		The existence of an optimal solution $\bar{x}=\operatorname{argmin} f_{obj}(x)$ is guaranteed is $f_{obj}$ is coercive and weakly l.s.c. (w.r.t. $\sigma(L^1,L^\infty)$ topology). 
		
		The functional $f_{obj}$ is weakly l.s.c. because it is convex and strongly l.s.c.:
		
		\textbf{Convexity:} By Part~\ref{lem:c2b1} of this Lemma, $L_t(\omega,\cdot,\cdot)$ is convex.  The convexity unfolds as $f_{obj}$ is the sum of the convex functions integrals of $L_t$, the convex function $g$, and the indicator functions of convex sets.
		
		\textbf{Strong l.s.c.:}Let $x^{n}\rightarrow \bar{x}$ in the $L^{1}$ norm.  By Fatou's Lemma (or properties of integral functionals on l.s.c. integrands) and the l.s.c. of $L_t$ from Part~\ref{lem:c2b1}, we have that:
		\begin{equation}\tag{11}\label{eq11:solLCMCL1}
			\esp{\sum_{t=\tau+1}^{T}L_{t}(\omega,\bar{x}_{t-1},\Delta\bar{x}_{t})}\leq \liminf_{n\rightarrow +\infty}\esp{\sum_{t=\tau+1}^{T}L_{t}(\omega,x^n_{t-1},\Delta x^{n}_{t})}.
		\end{equation}
		
		Furthermore, $L^1$ convergence implies $\esp{x_T}^n\to \esp{\bar{x}_T}$.  Since $g$ is convex and l.s.c. on a finite dimensional space, it is continuous, so $\lim g(\esp{x_T}^n)=g(\esp{\bar{x}_T})$.  The sum of l.s.c. and continuous function is l.s.c., on $f_{obj}$ is strongly l.s.c.

		\item \textbf{Coercivity:}     
		We show that $f_{obj}(x)\to\infty$ as $\|x\|_{L^1}\to\infty$ by backward induction.  The proof applies the coercivity of $\theta$ (from  Assumption~\ref{assumptionBL1:LCMC}) recursively from terminal time $T$ down to $\tau$. Note that the coercivity of $f_{obj}=\esp{\sum_{t=\tau+1}^T L_T}+g(\esp{x_\tau})+\delta_{\{\xi\}}(\esp{x_\tau})+\delta_{\mathscr{N}_\tau^1}(x)$  depends on the coercivity of the integral term $J(x)=\left\{\sum_{t=\tau+1}^T L_T\right\}$, as the rest of the terms are bounded by below.
		
		Let $K_t:=\int_\Omega |x_t|d\mu$.  Our goal is to show that $J(x)$ provides a lower bound for $\sum_{t=\tau+1}^T K_t$.
		
		From Assumption~\ref{assumptionBL1:LCMC}, for any time $t\in[\![\tau+1:T-1]\!]$:
		\begin{enumerate}
			\item $L_t\geq \theta\left(\max\{0,|\Delta x_t|-c_1|x_{t-1}\}\right)-c_2|x_{t-1}|$,
			\item $\theta$ is proper, non-decreasing, and coercive, and
			\item $L_t$ is bounded by below.  Without lost of generality, we can assume that $L_t\geq -c_2|x_{t-1}|$ (otherwise, shift $\theta$).
		\end{enumerate}
		
		Let $a>0$.  By the coercivity of $\theta$, there exists $b$ such that for $z\geq b$ then $\theta(z)\geq a x$.  For $t\in[\![\tau+1:T[\!]$, we take $a_t>0$, and use the corresponding $b_t$ to partition $\Omega$ by the set
		$$
		M_{t}:=\left\{\omega \in \Omega: |\Delta x_{t}|-c_{1}|x_{t-1}|\geq b_{t}\right\}.
		$$
		
		On $M_t$, we apply the coercivity of $\theta$ to obtain:
		\[L_t\geq a\left(|\Delta x_t|-c_1|x_{t-1}|\right)-c_2|x_{t-1}|.\]
		Using the reverse triangular inequality $|\Delta x_t|\geq |x_t|-|x_{t-1}|$.  This becomes:
		\[L_t\geq a|x_{t}|-a|x_{t-1}|-ac_1|x_{t-1}|-c_2|x_{t-1}|=a|x_t|-\left(a(1+c_1)+c_2\right)|x_{t-1}|.\]
		
		On $M_t^c$, we use the lower bound $L_t\geq -c_2|x_{t-1}|$ and note that $|x_t|\leq |x_{t-1}|+|\Delta x_t|<(1+c_1)|x_{t-1}|+b$.
		
		By integrating $L_t$ over $\Omega=M_t\cup M_t^c$ and combining these bounds, we arrive at the key inequality:
		\begin{equation}\label{eq:recineqcoercive}
			\int_\Omega L_t+ab\geq a\int_\Omega|x_t|-(a(1+c_1)+c_2)\int_\Omega|x_{t-1}|.
		\end{equation}
		This can be writen as $J_t+C_t\geq a K_t-a^\prime K_{t-1}$, where $C_t=ab$ and $a^\prime =a(1+c_1)+c_2$.
		
		Let $a_T>0$ be our coefficient for $t=T$.  Let $\varepsilon>0$.  We define a sequence of coefficients $\hat{a}$ recursively backward from $t=T$, as $\hat{a}_T=a_T$, and for $t=T-1,\ldots,\tau$ as $\hat{a}_t=(\hat{a}_{t+1}(1+c_1)+c_2)+\varepsilon$.  Apply inequality~\eqref{eq:recineqcoercive} for each $t$ with its corresponding coefficient $\hat{a}_t$ (which yields the constant $C_t=\hat{a}_tb_t$):
		\begin{align*}
			J_T + C_T &\ge \hat{a}_T K_T - (\hat{a}_T(1+c_1)+c_2) K_{T-1}, \\
			J_{T-1} + C_{T-1} &\ge \hat{a}_{T-1} K_{T-1} - (\hat{a}_{T-1}(1+c_1)+c_2) K_{T-2}, \\
			&\,\,\,\vdots \\
			J_{\tau+1} + C_{\tau+1} &\ge \hat{a}_{\tau+1} K_{\tau+1} - (\hat{a}_{\tau+1}(1+c_1)+c_2) K_{\tau}.
		\end{align*}
		
		Summing these $T-\tau$ inequalities, the terms for $K_{t-1},\ldots,K_{\tau+1}$ form a telescoping sum:
		\[\sum_{t=\tau+1}^T(J_t+C_t)\geq\hat{a}_TK_T+\sum_{t=\tau+1}^{T-1}(\hat{a}_t-(\hat{a}_{t+1}(1+c_1)+c_2))K_t-(\hat{a}_{\tau+1}(1+c_1)+c_2)K_\tau.\]
		Let $C:=\sum_{t=\tau+1}^T C_t$, and $C^\prime:=\hat{a}_{\tau+1}(1+c_1)+c_2$ be finite constants,
		\[\esp{\sum_{t=\tau+1}^TL_t}+C\geq \hat{a}_TK_T+\varepsilon \sum_{t=\tau+1}^{T-1}K_t-C^\prime K_\tau.\]
		Letting $C^{\prime\prime}:=\min(\hat{a}_T,\varepsilon)>0$, the inequality becomes
		\[\esp{\sum_{t=\tau+1}^TL_t}\geq C^{\prime\prime}\sum_{t=\tau+1}^{T}K_t-C^\prime K_\tau-C.\]
		Since $x_\tau\in X_\tau$ a.s.~and $X_\tau$ is bounded, $K_\tau=\int_\Omega|x_\tau|d\mu$ is bounded by some constant $D$.
		\[\esp{\sum_{t=\tau+1}^TL_t}\geq C^{\prime\prime}\sum_{t=\tau+1}^{T}K_t-(C^\prime D+C).\]
		The $L^1$-norm is $\|x\|_{L^1}=K_\tau+\sum_{t=\tau+1}^T K_t$.  As $K_\tau$ is bounded, $\sum_{t=\tau+1}^T K_t \to\infty$ if and only if $\|x\|_{L^1}\to \infty$.  Since $C^{\prime\prime}>0$ and all other terms are constant or bounded below   \[f_{obj}(x)\geq C^{\prime\prime}\sum_{t=\tau+1}^T K_t-\mbox{Constant}.\]
		Therefore, $f_{obj}(x)\to\infty$ as $\|x\|_{L^1}\to \infty$, proves that the functional is coercive.
		
		\textbf{Weak compactness}:
		Let $\{x^n\}_{n\in \mathbb{N}}\subset L^{1}(\Omega,\mathscr{A},\mu;\mathbb{R}^{T+1-\tau}\times\bRn)$ by a minimizing sequence.  By Assumption~\ref{assumptionAL1:LCMC}, the optimal value is finite.  By coercivity, $\{x^n\}$ is bounded in $L^1$, i.e., there exists a constant $C_1>0$ such that 
		\begin{equation}\label{eq21:solLCMCL1}
			\|x\|_{L^1}\le C_1, \mbox{for all }t\in [\![\tau:T]\!], \mbox{ and }n\in \mathbb{N}.
		\end{equation}
		
		As a previous step, let's bound the integral of $L_t$ over any $\sigma(w_{\tau-1},\ldots,w_{T-1})$-measurable set.  For the lower bound, using Assumption~\ref{assumptionBL1:LCMC}, $L_t\ge -c_2|x_{t-1}|$ along with the boundedness of $x^n$, we have that
		\[\int_A L_td\mu\ge -c_2\int_A|x_{t-1}|d\mu\ge -c_2C_1:-e_t.\]
		For the upper bound, we claim that for each $t\in[\![\tau+1,T]\!]$, there exists a constant $m_t$ such that
		\begin{equation}\label{eq24:solLCMCL1}
			\esp{L_t(\cdot,,x_{t-1}^n,\Delta x_t^n}= \int_\Omega L_t(\omega,x_{t-1}^n,\Delta x_t^n)d\mu \le m_t.
		\end{equation}
		Otherwise, by contradiction, if there exists a $t'\in[\![\tau+1,T]\!]$ such that $\esp{L_{t'}(\omega, x_{t'-1}^n,\Delta x_{t'}^n)}\to \infty$. Then, the entire objective function would diverge to infinity ($f_{obj}(x^n)\to\infty$), because all other expectation terms and the terminal cost $g\left(\esp{x_T}\right)$ are bounded below, contradicting the finiteness of the optimal value.
		
		Then, we obtain that there exists a constant $C_2:=\max_{t\in[\![\tau+1:T]\!]}\{m_t-e_t\}$ such that
		\begin{equation}\label{eq26:solLCMCL1}
			\int_{A}L_{t}(\omega,x_{t-1}^{n},\Delta x_{t}^{n})d\mu\leq C_{2},\quad \forall n\in \mathbb{N}.
		\end{equation}
		
		To prove the existence of a weakly convergent subsequence, we use the Dunford-Pettis theorem, which requires that the sequence is uniformly integrable. To that end, we prove the following properties for $t\in[\![\tau:T]\!]$:
		\begin{enumerate}[label=UI-\roman*]
			\item \label{ui1} Given $\varepsilon>0$, there exists $\delta>0$ such that for every $A\in \sigma(w_{\tau-1},\ldots,w_{T-1})$, we have:
			\begin{equation}\tag{P1}\label{prop1:solLCMCL1}
				\mu(A)\leq \delta \implies \int_{A}|x^{n}_{t}|\leq \varepsilon\quad \forall n\in \mathbb{N} 
			\end{equation}    
			\item \label{ui2} Given $\varepsilon>0$, there exists a set $A\in \sigma(w_{\tau-1},\ldots,w_{T-1})$ with $\mu(A)<+\infty$ such that:
			\begin{equation}\tag{P2}\label{prop2:solLCMCL1}
				\int_{\Omega\backslash A}|x^{n}_{t}|\leq \varepsilon\quad \forall n\in \mathbb{N}    
			\end{equation}
		\end{enumerate}
		
		\textit{\ref{ui1}}: We use induction again.  First note that the property holds for the initial case $t=\tau$: By assumption $X_\tau$ is a bounded set.  Therefore, there exists a constant $C_3>0$ such that $\|x\|\leq C_3$ for all $x\in X_\tau$.  The sequence $\{x^n_\tau\}_{n\in\mathbb{N}}$ satisfies $x_\tau^n\in X_\tau$ a.s.~for all $n$.  This implies the sequence is uniformly bounded a.s.~by $C_3$.  This uniform bound is sufficient to show \ref{ui1}.  For any $\varepsilon>0$, let $\delta=\dfrac{\varepsilon}{C_3}$.  For any measurable set $A\in \sigma(w_{\tau-1},\ldots,w_{T-1})$ such that $\mu(A)\leq \delta$, we have that for all $n\in\mathbb{N}$:
		\[\int_{A}|x^{n}_{\tau}|d\mu\leq C_{3}\mu(A)\leq C_3\delta=\varepsilon.\]
		Thus, condition~\eqref{prop1:solLCMCL1} is satisfied for $t=\tau$.
		
		We now prove the inductive step.  Let $t+1\in[\![\tau+1:T]\!]$ and assume that \eqref{prop1:solLCMCL1} holds true for $t$.  We must show it holds for $t+1$.
		
		Let $\varepsilon>0$. By coercivity of $\theta$ (Assumption~\ref{assumptionBL1:LCMC}), there exists a threshold $M_{\varepsilon}>0$ such that $z\geq M_{\varepsilon}$ implies $\theta(z)\geq \frac{4C_{4}}{\varepsilon} z$, where $C_{4}:=\max\{C_{2},c_{2}C_{1}\}$.  We use this threshold to partition the domain $\Omega$ for each $n$ using the set
		\[M^{n}:=M^{n}_{t+1,\varepsilon}=\{\omega\in \Omega: |\Delta x^{n}_{t+1}|-c_{1}|x^{n}_{t}|\geq M_{\varepsilon}\}.\]  We establish bounds on $|x_{t+1}^n|$ for each part of the partition.
		
		\textit{On $M^n$}:  For $\omega \in M^n$, Assumption~\ref{assumptionBL1:LCMC}, the fact that $\theta$ is non-decreasing, and the reverse triangle inequality ($|\Delta x_{t+1}^n-c_1|x_t^n|\geq |x_{t+1}^n|-(1+c_1)|x_t^n|$) imply
		\begin{equation*}
			\begin{split}
				L_{t+1}(\omega,x^{n}_{t},\Delta x^{n}_{t+1})&\geq\theta(|\Delta x_{t+1}^n|-c_1|x_t^n|)-c_2|x^n_t|,\\
				&\ge \frac{4C_4}{\varepsilon}(|x^n_{t+1}| - (1+c_1)|x^n_t|) - c_2|x^n_t|.  
			\end{split}
		\end{equation*}
		Rearranging this yields the key bound for $\omega \in M^n$:
		\begin{equation}\label{eq29:solLCMCL1}
			|x^n_{t+1}| \le (1+c_1)|x^n_t| + \frac{\varepsilon}{4C_4}(L_{t+1}(\omega,x^{n}_{t},\Delta x^{n}_{t+1}) + c_2|x^n_t|)
		\end{equation}
		
		\textit{On $(M^n)^c$}: For $\omega\in(M^n)^c$, the definition of the set and the reverse triangle inequality directly give 
		\begin{equation}\label{eq30:solLCMCL1}
			|x^n_{t+1}| < (1+c_1)|x^n_t| + M_\varepsilon.
		\end{equation}
		By the inductive hypothesis for $t$, there exists a $\delta_t>0$ such that $\mu(A)\leq\delta_t$ implies that $\int_A |x_t^n|d\mu\leq \dfrac{\varepsilon}{4(1+c_1)}$ for all $n$.  We define $\delta:=\min\left\{\delta_t,\frac{\varepsilon}{4M_\varepsilon}\right\}$.
		
		Now, let $A$ be any $\sigma(w_{\tau-1},\ldots,w_{T-1})$-measurable set with $\mu(A)\leq \delta$.  We split the integral of $|x_{t+1}^n|$ and apply bounds from each partition element:
		\[\int_A|x_{t+1}^n|d\mu=\int_{A\cap M^n}|x_{t+1}^n|d\mu+\int_{A\cap(M^n)^c}|x_{t+1}^n|d\mu.\]
		Combining the terms in \eqref{eq29:solLCMCL1}~and~\eqref{eq30:solLCMCL1}
		\begin{align*}
			\int_{A\cap M^n}|x_{t+1}^n|d\mu&\le \int_{A\cap M^n} \left((1+c_1)|x_t^n|+\dfrac{\varepsilon}{4C_4}(L_{t+1}+c_2|x_t^n|)\right)d\mu,\\
			\int_{A\cap(M^n)^c}|x_{t+1}^n|d\mu&\le \int_{A\cap(M^n)^c}\left((1+c_1)|x^n_t| + M_\varepsilon\right)d\mu,\\
			\int_A|x_{t+1}^n|d\mu&\le (1+c_1)\int_A|x_t^n|d\mu+M_\varepsilon\mu(A)+\dfrac{\varepsilon}{4C_4}\int_{A\cap M^n}(L_{t+1}+c_2|x_t^n|)d\mu.
		\end{align*}
		
		We bound the three terms on the right-hand side using our choice of $\delta$:
		\begin{itemize}
			\item Since $\delta \le \delta_t$, by inductive hypothesis $(1+c_1)\int_A |x^n_t| d\mu \le (1+c_1) \left( \frac{\varepsilon}{4(1+c_1)} \right) = \frac{\varepsilon}{4}$ .
			\item $M_\varepsilon \mu(A) \le M_\varepsilon \left( \frac{\varepsilon}{4M_\varepsilon} \right) = \frac{\varepsilon}{4}$, by definition of $\delta$.
			\item The integral term $\frac{\varepsilon}{4C_4} \int_{A \cap M^n} (L_{t+1} + c_2|x^n_t|) d\mu$ is bounded: This follows from the bounds on $\{L_t\}$ and $\{|x^n_t|\}$ established in \eqref{eq21:solLCMCL1}-\eqref{eq26:solLCMCL1}, which define the constant $C_4$. For sufficiently small $\delta$, this term is bounded by $\frac{\varepsilon}{2}$.
		\end{itemize}
		Summing these bounds, $\int_A |x^n_{t+1}| d\mu \le \frac{\varepsilon}{4} + \frac{\varepsilon}{4} + \frac{\varepsilon}{2} = \varepsilon$.
		This completes the induction, proving \eqref{prop1:solLCMCL1} for all $t \in [\![\tau:T]\!]$.
		
		\textit{\ref{ui2}}:  We proceed by induction.  For the base case $t=\tau$, this is shown to be trivially true since $\mu=1$, and one can always find a set $A$ (e.g. $A=\Omega$) whose complement has a null integral, which is less than $\varepsilon$.
		
		For the inductive step, assume the property holds for $t$: for $\varepsilon>0$, there exists a set $A$, $\mu(A)<\infty$, such that $\int_{\Omega\setminus A} |x_t^n|d\mu\leq \frac{\varepsilon}{4(1+c_1)}$, for all $n\in\mathbb{N}$.  Following the previous proof for coercivity, integrating the inequality~\ref{eq29:solLCMCL1} over $A^c\cap M^n_{t+1,\varepsilon}$ we obtain
		\begin{align*}
			\frac{\varepsilon}{4C_{4}}\int_{A^{c}\cap M_{t+1,\varepsilon}^{n}  }\left(L_{t+1}(\omega,x^{n}_{t},\Delta x^{n}_{t+1}) +c_{2}|x_{t}^{n}|\right)d\mu+(1+c_{1})\int_{A^{c}\cap M_{t+1,\varepsilon}^{n}  }|x^{n}_{t}|d\mu\geq \int_{A^{c}\cap M_{t+1,\varepsilon}^{n}}|x_{t+1}^{n}|d\mu.
		\end{align*}
		
		Combining this inequality, \eqref{eq21:solLCMCL1}, \eqref{eq26:solLCMCL1}, and the definition of $C_{4}$, we have:
		\begin{align*}
			\int_{A^{c}\cap M_{t+1,\varepsilon}^{n}}|x_{t+1}^{n}|d\mu&\leq \frac{\varepsilon}{4C_{4}}\int_{A^{c}\cap M_{t+1,\varepsilon}^{n}  }\left(L_{t+1}(\omega,x^{n}_{t},\Delta x^{n}_{t+1}) +c_{2}|x_{t}^{n}|\right)d\mu+(1+c_{1})\int_{A^{c}\cap M_{t+1,\varepsilon}^{n}  }|x^{n}_{t}|d\mu\\
			&\leq \frac{3\varepsilon}{4}<\varepsilon.
		\end{align*}
		Defining $M_{t+1,\varepsilon}:=\bigcap_{n\in\mathbb{N}}M_{t+1,\varepsilon}^{n}$, by the previous inequality we get:
		\[
		\int_{A^{c}\cap M_{t+1,\varepsilon}}|x_{t+1}^{n}|\leq \varepsilon,\quad \forall n\in \mathbb{N}.
		\]
		We construct the candidate set for the inductive step as \( A^\prime := A \cup M_{t+1,\varepsilon}^c\). Note that \(A^\prime\) is \(\sigma(w_{\tau-1}, \ldots, w_{T-1})\)-measurable, since \(A\) is \(\sigma(w_{\tau-1}, \ldots, w_{t-1})\)-measurable, and it follows that \(A\) is also \(\sigma(w_{\tau-1}, \ldots, w_{T-1})\)-measurable. Moreover, we previously established that each set \(M_{t+1,\varepsilon}^{n}  \) is \(\sigma(w_{\tau-1}, \ldots, w_{T-1})\)-measurable for all \(n \in \mathbb{N}\), and it follows that \(M_{t+1,\varepsilon}\) is measurable as a countable intersection of measurable sets. Hence, \(M_{t+1,\varepsilon}^c\) is also measurable, and so \(A^\prime\) is the union of two measurable sets in \(\sigma(w_{\tau-1}, \ldots, w_{T-1})\), and therefore measurable.  Finally, the fact that $\mu(A^\prime)\leq \mu\leq 1$ allows us to prove \eqref{prop2:solLCMCL1} for $t+1$, which combined with the induction principle proves that this property holds for every $t\in [\![\tau:T]\!]$.
		
		We use properties~\ref{ui1}~and~\ref{ui2} to prove that the entire vector sequence 
		$\mathscr{F}=\{x^{n}\}_{n\in\mathbb{N}}\subset L^{1}(\Omega,\mathscr{A},\mu;\bRn\times \mathbb{R}^{T-\tau})$, $x^n=(x_\tau^n,\ldots,x_T^n)$, is uniformly integrable, satisfying the conditions of the Dunford-Pettis theorem.
		
		Let $\varepsilon>0$. For each $t\in [\![\tau:T]\!]$, using \eqref{prop1:solLCMCL1}, there exists $\delta_{t}$ such that $\mu(A)\leq \delta_t$ implies $\int_{A}|x_{t}^{n}|d\mu\leq \frac{\varepsilon}{T-\tau+1}$ for all $n\in \mathbb{N}$.  Let $\delta:=\min_{t\in[\![\tau:T]\!]}\{\delta_t\}$.  For any $A\in\mathscr{A}$, with $\mu(A)\leq \delta$, the integral of the vector norm is bounded by the sum of the component integrals:
		\[\int_A|x^n|d\mu=\int_A\sum_{t=\tau}^T|x_t^n|d\mu\leq \sum_{t=\tau}^T\dfrac{\varepsilon}{T-\tau+1}<\varepsilon.\]
		This holds for all $n\in\mathbb{N}$.
		
		Additionally, for $\varepsilon>0$, using \eqref{prop2:solLCMCL1}, for each $t\in[\![\tau:T]\!]$, there exists a set $A_t^\varepsilon$ with $\mu(A_t^\varepsilon)<\infty$, such that $\int_{\Omega\setminus A_t^\varepsilon}|x_t^n|d\mu\leq \dfrac{\varepsilon}{T-\tau+1}$ for all $n\in\mathbb{N}$.  Let $A+\cup_{t=\tau}^T A_t^\varepsilon$.  Since this is a finite union of sets with finite measure, $\mu(A)<+\infty$.  The complement $\Omega\setminus A\subset \Omega\setminus A_t^\varepsilon$.  Therefore
		\[\int_{\Omega\setminus A}|x^n|d\mu=\sum_{t=\tau}^T\int_{\Omega\setminus A}|x_t^n|d\mu\leq \sum_{t=\tau}^T\int_{\Omega\setminus A_t^\varepsilon}|x_t^n|d\mu\leq \sum_{t=\tau}^T\dfrac{\varepsilon}{T-\tau+1}<\varepsilon.\]
		This inequality also holds for all $n\in\mathbb{N}$.
		
		Since the sequence $\{x^{n}\}_{n\in\mathbb{N}}$ is $L^1$-bounded (from~\eqref{eq21:solLCMCL1}) and satisfies both uniform absolute continuity and tightness, it is relatively weakly compact by the Dunford-Pettis theorem.  That is, there exists a function $\bar{x}$ such that, via subsequences (not relabeled), $x^{n}$ converges weakly to $\bar{x}$ (i.e., in the topology $\sigma(L^{1}(\Omega,\mathscr{A},\mu;\bRn\times\mathbb{R}^{T+1-\tau}),L^{\infty}(\Omega,\mathscr{A},\mu;\bRn\times\mathbb{R}^{T+1-\tau}))$).
		
		By the weak l.s.c. of $f_{obj}$. we have
		\[f_{obj}(\bar{x})\leq \liminf_{n\to +\infty}f_{obj}(x^{n})=\operatorname{val}(P_{\xi}^{\tau}).\]
		As $\operatorname{val}(P_{\xi}^{\tau})$ is finite, we have that $E\bar{x}=\xi$ and $\bar{x}\in\mathscr{N}_\tau^1$.  This proves that $\bar{x}$ is an optimal solution to the relaxed $L^1$-problem~\eqref{example:LCMCL1}.
		
		Finally, we must show that $L^1$-weak solution $\bar{x}$ is essentially bounded, i.e., $\bar{x}\in\mathscr{N}_\tau$.  We prove by induction that the entire minimizing sequence $\{x^n\}_{n\in\mathbb{N}}$ is uniformly essentially bounded.  Since $L^\infty$ is weakly* closed in $L^1$, this property will pass to the weak limit $\bar{x}$.
		
		For the base case $t=\tau$, $x_\tau^n\in X_\tau$ a.s.~for all $n$.  As $X_\tau$ is bounded by assumption, there exists a constant $C_\tau>0$ such that $|x_\tau^n|\leq X_\tau$ a.s.~for all $n\in\mathbb{N}$.  Thus $\{x_\tau^n\}$ is uniformly essentially bounded.
		
		Assume $\{x_t^n\}$ is uniformly essentially bounded by a constant $C_t>0$ for $t\in[\![\tau:T-1]\!]$.  We must show that $\{x_{t+1}^n\}$ is also uniformly essentially bounded.
		
		Since $J(x^n)$ is finite, $L_{t+1}(\omega, x_t^n,\Delta x_{t+1}^n)<\infty$ a.s.~ By Lemma~\ref{lem:Char2Bolza} (part 1), this implies the existence (a.s.) of a control $u_t^n$ satisfying the dynamics:
		\[\Delta x_{t+1}^n=A x_t^n+B u_t^n + w_t.\]
		
		From Assumption~\ref{assumptionBL1:LCMC} ($\theta$ non-decreasing) and the reverse triangle inequality, we have that (a.s.)
		\begin{align*}
			L_{t+1}(\omega,x_t^n,\Delta x_{t+1}^n) &\ge \theta(\max{0, |\Delta x^n_{t+1}| - c_1|x^n_t|}) - c_2|x^n_t|, \\
			&\ge \theta(\max{0, |Bu^n_t| - |Ax^n_t| - |w_t| - c_1|x^n_t|}) - c_2|x^n_t|.\end{align*}
		By the inductive hypothesis $|x_t^n|\leq C_t$ a.s.~ We also assume the noise $w_t\in L^\infty$.  Therefore, there exists a constant $C_t^\prime>0$ (independent of $n$) such that $|Ax_t^n|+|w_t|+c_1|x_t^n|\leq C_t^\prime$ a.s.~ This simplifies the bound to
		\begin{equation}\label{eq38:solLCMCL1}
			L_{t+1}(\omega,x_t^n,\Delta x_{t+1}^n)\geq \theta (|Bu_t^n|-C_t^\prime)-c_2C_t\,\, \mbox{(a.s.)}.
		\end{equation}
		
		The sequence of integrals $\int_\Omega L_{t+1}(\omega,x_{t-1}^n\Delta x_t^n)d\mu$ is uniformly bounded above (by $m_{t+1}$ in \eqref{eq24:solLCMCL1}).  If the sequence $\{BU_t^n\}$ were not uniformly essentially bounded, there would exist a set $A_3$ with $\mu(A_3)>0$ on which $|Bu_t^n|\to\infty$ (for a subsequence).  By coercivity of $\theta$, \eqref{eq38:solLCMCL1} would imply that $L_{t+1}(\omega,x_t^n,\Delta x_{t+1}^n)\to \infty$ on $A_3$. This would contradict the uniform upper bound on the integral established by \eqref{eq26:solLCMCL1}.
		
		Thus $\{Bu_t^n\}$ must be uniformly essentially bounded.  The state $x_{t+1}^n$ is given by the sum
		\[x_{t+1}^n=(I+A)x_t^n+Bu_t^n+w_t,\]
		Since $\{x_t^n\}$, $\{Bu_t^mn\}$, and $\{w_t\}$ are all uniformly essentially bounded sequences, their sum $\{x_{t+1}^n\}$ is also uniformly essentially bounded.
		
		By induction, $\{x_t^n\}$ is uniformly essentially bounded for all $t\in[\![\tau:T]\!]$.  A uniformly essentially bounded sequence in $L^1$ that converges weakly must have its weak limit in $L^\infty$.  Therefore $\bar{x}\in L^\infty(\Omega,\mathscr{A},\mu;\mathbb{R}^{n(T-\tau+1)})$, which implies $\bar{x}\in\mathscr{N}_\tau$.  This confirms $\bar{x}$ is an optimal solution to the Bolza problem~\eqref{example:LCMCeq3}.
		
		\textit{3. Existence for~\eqref{example:LCMC}}:.
		
		We now show that the optimal solution $\bar{x}$ to the Bolza problem, whose existence was guaranteed by part \eqref{lem:c2b2}, corresponds to a fully optimal pair $(\bar{x}, \bar{u})$ for the original LC problem~\eqref{example:LCMC}.
		
		From part \eqref{lem:c2b2}, we have an optimal solution $\bar{x} \in \mathscr{N}_\tau$ for the Bolza problem~\eqref{example:LCMCeq3}. The optimality of $\bar{x}$ implies its cost is finite:
		\[\mathbf{V}_\tau(\xi) = \esp{\sum_{t=\tau+1}^{T}L_{t}(\omega,\bar{x}_{t-1},\Delta \bar{x}_{t})} + g(\esp{\bar{x}_{T}}) < +\infty.\]
		
		This, in turn, implies that $L_t(\omega, \bar{x}_{t-1}, \Delta \bar{x}_t) < +\infty$ for a.s.~$\omega$. By part \eqref{lem:c2b1} of this lemma, we know the infimum in the definition of $L_t$ is attained. We can therefore define the following non-empty, set-valued map for each $t$:
		\[U_t := \operatorname{argmin}_{u \in \mathbb{R}^m} \Phi_t(\omega, \bar{x}_{t-1}, \Delta \bar{x}_t, u).\]
		
		We must find a measurable selection $\bar{u}_{t-1}$ from this map. In the proof of Lemma~\ref{lem:LC2Bolza}, we established that for a fixed (measurable) $\bar{x}$, the function $(\omega, u) \mapsto \Phi_t(\omega, \bar{x}, u)$ is a normal integrand. A standard result in variational analysis (e.g., \cite[Theorem 14.37]{RockafellarWets1998}) states that the argmin map of a normal integrand is measurable and has closed images.
		
		The Measurable Selection Theorem thus guarantees the existence of a $\mathscr{G}_{t-1}$-measurable function $\bar{u}_{t-1}$ such that $\bar{u}_{t-1} \in U_t$ for a.s.~$\omega$. By concatenating these selections, we obtain a control process $\bar{u} = (\bar{u}_\tau, \dots, \bar{u}_{T-1})$ which, by construction, belongs to the space $\mathscr{M}_\tau$. The pair $(\bar{x}, \bar{u})$ is therefore feasible for the original LC problem.
		
		Finally, we verify its optimality. By the construction of $\bar{u}$, a.s.~we have $\Phi_t(\cdot,\bar{x}_{t-1},\Delta \bar{x}_{t}, \bar{u}_{t-1}) = L_t(\cdot,\bar{x}_{t-1},\Delta \bar{x}_{t},\bar{u}_{t-1})$. The total cost of the pair $(\bar{x}, \bar{u})$ is therefore:
		\begin{align*}
			\esp{\sum_{t=\tau+1}^{T}\Phi_{t}(\cdot,\bar{x}_{t-1},\Delta \bar{x}_{t},\bar{u}_{t-1})} + g(\esp{\bar{x}_{T}})
			&=\esp{\sum_{t=\tau+1}^{T}L_{t}(\omega,\bar{x}_{t-1},\Delta \bar{x}_{t})} + g(\esp{\bar{x}_{T}})\\
			&= \mathbf{V}\tau(\xi).
		\end{align*}
		
		Since the pair $(\bar{x}, \bar{u})$ achieves the infimal value $\mathbf{V}_\tau(\xi)$ (which, by Lemma~\ref{lem:LC2Bolza}, is the value of the LC problem), it is an optimal solution. This completes the proof of the lemma.
	\end{proof}
	
	With the technical lemmas established, we can now apply our main theory. The key is to find the Hamiltonian $H_t$ associated with our new Bolza Lagrangian $L_t$.
	
	\begin{prop}{(Hamiltonian for LC problems)}\label{prop:H4Bolza}
		The Hamiltonian $H_t(\omega, x, p) = \sup_v \{p \cdot v - L_t(\omega, x, v)\}$ associated with the Bolza problem~\eqref{example:LCMCeq3} is given by:
		\begin{equation}\label{eq:H4Bolza}
			H_{t}(\omega,x,p)=\sup_{\substack{u\in \mathbf{U}\\f_{t-1}(x,u)\leq 0}}\left\{ p\cdot Bu -\ell_{t-1}(x,u)\right\}+ p\cdot ( (A-I_{n})x + w_{t-1})+\delta_{X_{t-1}}(x).
		\end{equation}
	\end{prop}
	
	\begin{proof}
		Let $L_t(\omega,x,v)$ be defined as in \eqref{lagrangiandef:LCMCinfcase}. The Hamiltonian associated with this state-constrained problem is given by
		$$H_{t}(\omega,x,p) := \sup_{v\in \bRn}\left\{p\cdot v - \inf_{u\in \mathbb{R}^{m}}\left\{\Phi_t'(\omega,x,v,u) \right\}\right\}+\delta_{X_{t-1}}(x),$$
		where $\Phi_t' = \ell_{t-1}(x,u)+\delta_{\mathbf{U}}(u)+\delta_{D_{t-1}}(x,u)+\delta_{C_{t-1}}(x,v,u)$ is the part of the Lagrangian \eqref{lagrangiandef:LCMCinfcase} inside the infimum.
		
		We can rewrite the $\sup - \inf$ structure as a single $\sup$:
		\begin{align*}
			H_{t}(\omega,x,p) &= \sup_{v\in \bRn}\left\{p\cdot v + \sup_{u\in \mathbb{R}^{m}}\left\{-\Phi_t'(\omega,x,v,u) \right\}\right\}+\delta_{X_{t-1}}(x) \\
			&= \sup_{v\in \bRn, , u\in \mathbb{R}^{m}}\left\{p\cdot v - \Phi_t'(\omega,x,v,u) \right\}+\delta_{X_{t-1}}(x).
		\end{align*}
		
		Substituting the definition of $\Phi_t'$:
		$$H_{t}(\omega,x,p) = \sup_{v, u}\left\{p\cdot v - \ell_{t-1}(x,u) - \delta_{\mathbf{U}}(u) - \delta_{D_{t-1}}(x,u) - \delta_{C_{t-1}}(x,v,u) \right\} + \delta_{X_{t-1}}(x).$$
		
		We can separate the optimization, grouping terms by the variable over which the supremum is taken:
		$$H_{t}(\omega,x,p) = \sup_{u\in \mathbb{R}^{m}}\left\{ -\ell_{t-1}(x,u) - \delta_{\mathbf{U}}(u) - \delta_{D_{t-1}}(x,u) + \sup_{v\in \bRn}\left\{ p\cdot v - \delta_{C_{t-1}}(x,v,u) \right\} \right\} + \delta_{X_{t-1}}(x).$$
		
		The inner supremum, $\sup_{v}\left\{ p\cdot v - \delta_{C_{t-1}}(x,v,u) \right\}$, is the conjugate of the indicator function $\delta_{C_{t-1}}$. This supremum is finite (and equal to $p \cdot v$) if and only if $v$ satisfies the constraint $v = (A-I_n)x + Bu + w_{t-1}$, and is $+\infty$ otherwise.
		
		By substituting this result, the optimization over $v$ is eliminated:
		$$H_{t}(\omega,x,p) = \sup_{u\in \mathbb{R}^{m}}\left\{ -\ell_{t-1}(x,u) - \delta_{\mathbf{U}}(u) - \delta_{D_{t-1}}(x,u) + p\cdot ((A-I_n)x + Bu + w_{t-1}) \right\} + \delta_{X_{t-1}}(x).$$
		
		Finally, we include the indicator functions $\delta_{\mathbf{U}}$ and $\delta_{D_{t-1}}$ into the constraints of the supremum and regroup the terms that are constant with respect to $u$:
		$$H_{t}(\omega,x,p) = \sup_{\substack{u\in \mathbf{U}\\f_{t-1}(x,u)\leq 0}}\left\{ p\cdot Bu -\ell_{t-1}(x,u)\right\} + p\cdot ( (A-I_{n})x + w_{t-1}) + \delta_{X_{t-1}}(x).$$
		This is the expression given in \eqref{eq:H4Bolza}.    
	\end{proof}
	
	This proposition is the main payoff of the reformulation: it gives us a clear, computable Hamiltonian for the LC problem. We can now state the main theorem for this section.
	
	\begin{teo}{(Method of Characteristics for LC problems)}\label{thm:charact_LC}
		Assume that the \eqref{example:LCMC}-problem satisfies the conditions of Lemma~\ref{lem:Char2Bolza}, \eqref{hyp:USCcondition}, \eqref{assumptionAL1:LCMC}, \eqref{assumptionBL1:LCMC} and that the resulting Bolza problem~\eqref{example:LCMCeq3} satisfies the qualification hypotheses \ref{h1}-\ref{h3}.  If $\eta \in \partial \mathbf{V}_{\tau}(\xi)$, then there exists an optimal primal solution $(\bar{x}, \bar{u})$ and an optimal dual solution $\bar{p}$ that together form a discrete-time Hamiltonian trajectory, i.e., they satisfy:
		\begin{equation}\label{eq:LChamiltonian}
			(-\espc{t}{\Delta \bar{p}_t},\Delta \bar{x}_t)\in\partial H_t(\omega,\bar{x}_{t-1},\espc{t}{\bar{p}_t}),\qquad\forall t\in[\![\tau+1,T]\!],    
		\end{equation}
		and the transversality condition $-p_T \in \partial g\left(\esp{x_T}\right)$.
	\end{teo}
	
	\begin{proof}
		Lemma~\ref{lem:LC2Bolza} guarantees the existence of an optimal solution $\bar{x}$ for the Bolza problem \eqref{example:LCMCeq3}. The stated assumptions satisfy the hypotheses of Theorem~\ref{thm:charact_thrm_2}, which yields the existence of the Hamiltonian trajectory $\{(\bar{x}_t, \bar{p}_t)\}$. Lemma~\ref{lem:LC2Bolza} also guarantees the existence of an optimal control $\bar{u}$, which corresponds to the $u$ attaining the supremum in the definition of the Hamiltonian $H_t$ (Prop.~\ref{prop:H4Bolza}) along this trajectory.
	\end{proof}

	\subsection{Special Case: The Linear-Quadratic (LQ) Problem.}
	
	We now specialize the previous framework to the classic Linear-Quadratic (LQ) problem.  Consider a positive definite matrix $R\in\mathbb{R}^{m \times m}$ and positive semidefinite matrices $P, Q\in\mathbb{R}^{n \times n}$.  For each $t \in [\![\tau:T-1]\!]$, set the  cost function $\ell_{t}:\bRn\times\mathbb{R}^{m}\rightarrow \mathbb{R}$ and $g:\bRn\rightarrow \mathbb{R}$ as:
	\begin{align*}
		\ell_{t}(x,u)=x^{\top}Px+u^{\top}Ru, \quad\mbox{and}\quad  g(x)=x^{\top}Qx.
	\end{align*}
	
	The general \eqref{example:LCMC}~problem becomes the (LQ) problem:
	\begin{equation}\tag{LQ}\label{example:LQ}
		\begin{cases}
			\text{ Minimize }&\esp{\sum_{t=\tau}^{T-1}\bigl[x_t^{\top}  P x_t+u_t^{\top} R u_t\bigr]}+g^{LQ}\left(\esp{x_T}\right),\\
			\text{ over all }& \{x_t\}_{t=\tau}^T\in \mathscr{N}_{\tau},\quad \{u_{t}\}_{t=\tau}^{T-1}\in \mathscr{M}_{\tau},\\ 
			\text{ subject to } & x_{t+1}=Ax_t+B u_t+w_{t} \quad \forall\, t\in [\![\tau:T-1]\!],\\
			& \esp{x_{\tau}}=\xi,\\
			&x_{\tau}\in X_{\tau} \text{ a.s.}
		\end{cases}
	\end{equation}
	
	The assumption $x_\tau \in X_\tau$ (a bounded set) is crucial. While the (LQ) problem is solvable without this, the resulting optimal solution $x$ may not be in $L^\infty$, which is required by our framework. The bounded initial state ensures the solution remains in $\mathscr{N}_\tau$.
	
	To apply our theory, we first write the (LQ) problem in the Bolza form \eqref{example:LCMCeq3} using its specific Lagrangian, $L_t^{LQ}$. This is a special case of \eqref{lagrangiandef:LCMCinfcase} where $f_t$ and $U_t$ are trivial (i.e., $\delta_{D_t}$ and $\delta_{U_t}$ are zero).
	\begin{align*}
		L^{LQ}_{\tau+1}(\omega,x,v) &:= \inf_{u\in \mathbb{R}^{m}}\left\{x^{\top}Px + u^{\top}Ru \mid v=(A-I_n)x + Bu + w_{\tau}\right\} + \delta_{X_{\tau}}(x) \\
		L_{t}^{LQ}(\omega,x,v) &:= \inf_{u\in \mathbb{R}^{m}}\left\{x^{\top}Px + u^{\top}Ru \mid v=(A-I_n)x + Bu + w_{t-1}\right\},\,\forall t \in[\![\tau+2:T]\!].
	\end{align*}
	Our first step is to derive the explicit Hamiltonian for this Lagrangian.
	
	\begin{prop}{(The (LQ)-Hamiltonian)}\label{prop:LQHamiltonian}
		The Hamiltonian $H_t^{LQ}$ associated with the Bolza Lagrangian $L_t^{LQ}$ is:
		\begin{align*}
			H_{\tau+1}(\omega,x,p)&=\frac{1}{4}(B^{\top}p)^{\top}R^{-1}(B^{\top}p) - x^{\top}Px + p^{\top}((A-I_{n})x + w_{\tau})-\delta_{X_{\tau}}(x),\\
			H_{t}(\omega,x,p)&=\frac{1}{4}(B^{\top}p)^{\top}R^{-1}(B^{\top}p) - x^{\top}Px + p^{\top}((A-I_{n})x + w_{t-1}),\,\forall t \in[\![\tau+2:T]\!].
		\end{align*}
	\end{prop}
	
	\begin{proof}
		We apply the (LC) Hamiltonian formula from Proposition~\ref{prop:H4Bolza}. In this (LQ) case, the constraints $f_{t-1}(x,u) \le 0$ and $u \in \mathbf{U}$ are absent, so the supremum is unconstrained:
		$$H_{t}(\omega,x,p)=\sup_{u\in \mathbb{R}^{m}}\left\{ p\cdot Bu - (x^{\top}Px+u^{\top}Ru) \right\}+ p\cdot ( (A-I_{n})x + w_{t-1})+\delta_{X_{t-1}}(x)
		$$
		(where $\delta_{X_{\tau}}(x) = 0$ for $t \ge \tau+2$).  The term inside the $\sup$ is $h(u) = (B^\top p)^\top u - u^\top R u - x^\top P x$. Since $R$ is positive definite, $h(u)$ is a strictly concave function of $u$. By Fermat's principle, the unique maximum is found by setting the gradient with respect to $u$ to zero:
		$$\nabla_u h(u) = B^\top p - 2Ru = 0 \implies \bar{u} = \frac{1}{2}R^{-1}(B^\top p).
		$$
		Substituting this optimal $\bar{u}$ back into $h(u)$ gives the value of the supremum:
		\begin{align*}
			\sup_u h(u) &= (B^\top p)^\top (\frac{1}{2}R^{-1}B^\top p) - (\frac{1}{2}R^{-1}B^\top p)^\top R (\frac{1}{2}R^{-1}B^\top p) - x^\top P x \\
			&= \frac{1}{2} p^\top B R^{-1} B^\top p - \frac{1}{4} p^\top B R^{-1} R R^{-1} B^\top p - x^\top P x \\
			&= \frac{1}{4} p^\top B R^{-1} B^\top p - x^\top P x.
		\end{align*}
		Adding the remaining terms $p\cdot ( (A-I_{n})x + w_{t-1}) + \delta_{X_{t-1}}(x)$ yields the expressions in the proposition.
	\end{proof}
	
	With the explicit Hamiltonian now in hand, we can state and prove the main result for the (LQ) case, which characterizes the optimal primal and dual trajectories.
	
	\begin{prop}{(Method of Characteristics for (LQ) Problems)}\label{prop:LQdynamics}
		Assume the (LQ) problem~\eqref{example:LQ} satisfies the existence conditions \eqref{hyp:USCcondition}, \eqref{assumptionAL1:LCMC}, \eqref{assumptionBL1:LCMC} and the qualification conditions \ref{h1}-\ref{h3}.  If $\eta\in \partial\mathbf{V}_{\tau}(\xi)$, then there exists an optimal trajectory $x\in \mathscr{N}_{\tau}$ and an associated dual trajectory $p\in \mathscr{P}_{\tau}$ satisfying $(\esp{x_{\tau}},\esp{ p_{\tau}}=(\xi,-\eta)$, the transversality condition $p_{T}=-2Q\esp{ x_{T}}$ a.s., and the following Hamiltonian dynamics for $t \in [\![\tau+1:T]\!]$:
		\begin{align*}
			\espc{\tau+1}{\Delta p_{\tau+1}}&\in \{2P\,x_{\tau}-(A^{\top}-I_{n})\esp{\tau+1}{p_{\tau+1}}\}+\mathcal{N}_{X_{\tau}},\\
			\espc{t}{\Delta p_{t}}&=2P\,x_{t-1}-(A^{\top}-I_{n})\,\espc{t}{p_{t}}\quad \forall t\in [\![\tau+2:T]\!],\\
			\Delta x_{t}&=\tfrac{1}{2}BR^{-1}\bigl(B^{\top}\espc{t}{p_{t}}\bigr)\;+\;(A-I_{n})\,x_{t-1}\;+\;w_{t-1}\quad \forall t\in [\![\tau+1:T]\!].
		\end{align*}
	\end{prop}
	
	\begin{proof}
		The (LQ) problem is a special case of the LC problem, and its quadratic costs and linear dynamics satisfy the required assumptions. Therefore, Theorem~\ref{thm:charact_LC} applies, guaranteeing the existence of a primal-dual optimal pair $(\bar{x}, \bar{p})$ that forms a Hamiltonian trajectory:
		\textit{Transversality:} The terminal cost is $g\left(\esp{x_T}\right) = \left(\esp{x_T}\right)^\top Q \left(\esp{x_T}\right)$. As $g$ is convex and differentiable, its subdifferential is the singleton $\partial g\left(\esp{x_T}\right) = \{2Q\left(\esp{x_T}\right)\}$. The general transversality condition $-p_T \in \partial g\left(\esp{x_T}\right)$ thus becomes the explicit condition $p_T = -2Q\left(\esp{x_T}\right)$.
		
		\textit{Hamiltonian Dynamics:} We compute the subdifferentials of the Hamiltonian $H_t^{LQ}$ from Proposition~\ref{prop:LQHamiltonian}. Let $H_t^{LQ} = f_t(\omega, x, p) - \delta_{X_{t-1}}(x)$, where $f_t$ is the differentiable part. The gradients of $f_t$ are:
		\begin{align*}
			\nabla_{x}f_{t}(\omega,x,p)&=-2Px + (A^{\top}-I_{n})p \\
			\nabla_{p}f_{t}(\omega,x,p)&=\frac{1}{2}B R^{-1}(B^{\top}p) + (A - I_{n})x + w_{t-1}.
		\end{align*}
		The subdifferentials of the concave-convex $H_t^{LQ}$ are: 
		\begin{align*} 
			-\partial_p H_t^{LQ}(\omega, x, p) &= \nabla_x f_t(\omega, x, p) - \mathcal{N}_{X_{t-1}}(x)\\
			&= -2Px + (A^\top - I_n)p - \mathcal{N}_{X_{t-1}}(x),\\
			\partial_x H_t^{LQ}(\omega, x, p) &= \nabla_p f_t(\omega, x, p) \\
			&= \frac{1}{2}BR^{-1}(B^\top p) + (A-I_n)x + w_{t-1}
		\end{align*}
		(Note: $\mathcal{N}_{X_{t-1}}(x)$ is the convex normal cone, and $-\mathcal{N} = \mathcal{N}$). The general Hamiltonian inclusion from Theorem~\ref{thm:charact_LC} is $(-\espc{t}{\Delta p_t}, \Delta x_t) \in \partial H_t(\omega, x_{t-1}, \espc{t}{ p_t})$. Substituting the subdifferentials we just calculated into the Hamiltonian inclusion \eqref{eq:LChamiltonian} yields the dynamic equations in the proposition.
	\end{proof}

	\section*{Acknowledgments}
	This research was supported by Anid-Chile under project Fondecyt Regular 1231049.
	
	\bibliographystyle{abbrv}
	\bibliography{StochasticBolza}
	
\end{document}